\documentclass[10pt]{amsart}
\usepackage[utf8]{inputenc}
\usepackage[T1]{fontenc}
\usepackage[english]{babel}
\title{An embedding of the Morse boundary in the Martin boundary}
\author{Matthew Cordes and Matthieu Dussaule and Ilya Gekhtman}
\date{}

\usepackage{silence}
\usepackage{comment}
\usepackage{amsmath}
\usepackage{amssymb}
\usepackage{dsfont}
\usepackage{amsfonts}
\usepackage{amsthm}
\usepackage{tikz}
\usepackage{graphicx}
\usepackage{fancyhdr}
\usepackage{caption}
\usepackage{enumerate}
\usepackage{mathtools}
\usepackage[colorlinks=true,pdfstartview=FitH,bookmarks=false]{hyperref}
\hypersetup{urlcolor=black,linkcolor=black,citecolor=black,colorlinks=true}

\newcommand\R{\mathbb{R}}

\theoremstyle{plain}
\newtheorem{definition}{Definition}[section]
\newtheorem{proposition}[definition]{Proposition}
\newtheorem{corollary}[definition]{Corollary}
\newtheorem{theorem}[definition]{Theorem}
\newtheorem*{theorem*}{Theorem}

\newtheorem*{question*}{Question}
\newtheorem{lemma}[definition]{Lemma}
\newtheorem*{prop*}{Proposition}
\newtheorem*{lem*}{Lemma}

\theoremstyle{remark}

\newtheorem*{rem*}{Remark}

\usepackage{etoolbox}
\apptocmd{\sloppy}{\hbadness 10000\relax}{}{}
\apptocmd{\sloppy}{\vbadness 10000\relax}{}{}

\usepackage{tikz}

\begin{document}
\begin{abstract}
We construct a one-to-one continuous map from the Morse boundary of a hierarchically hyperbolic group to its Martin boundary.
This construction is based on deviation inequalities generalizing Ancona's work on hyperbolic groups \cite{AnconaMartinhyperbolic}.
This provides a possibly new metrizable topology on the Morse boundary of such groups.
We also prove that the Morse boundary has measure 0 with respect to the harmonic measure unless the group is hyperbolic.
\end{abstract}

\maketitle

\section{Introduction}

To any Gromov hyperbolic space can be associated a compactification obtained by gluing asymptotic equivalence classes of geodesic rays \cite{Gromov}.
This topological space, called the Gromov boundary, is a quasi-isometry invariant and has therefore been invaluable in studying algebraic, geometric, probabilistic and dynamical properties of hyperbolic groups, that is groups which act properly and cocompactly on Gromov hyperbolic spaces.

A major theme in recent research in metric geometry and geometric group theory has been studying various generalizations of hyperbolic groups which nevertheless admit interesting actions on Gromov hyperbolic spaces.  
These include
\begin{enumerate}[(a)]
\item Weakly hyperbolic groups: groups which admit non-elementary actions on (possibly non proper) Gromov hyperbolic spaces.
\item Acylindrically hyperbolic groups \cite{Sela}, \cite{Bowditchcurvecomplex}, \cite{Osinacylindrical}: weakly hyperbolic groups which admit a non-elementary action on a hyperbolic space satisfying a weakened discontinuity assumption called acylindricity.
These include outer automorphism groups of free groups, as well as the classes listed below.
\item Hierarchically hyperbolic groups \cite{BehrstockHagenSisto1}, \cite{Sistosurvey}: a special class of weakly hyperbolic groups admitting a nice combinatorial description.
These include mapping class groups, graph products such a right angled Artin and Coxeter groups, and finite covolume Kleinian groups.
\item Relatively hyperbolic groups \cite{Farb}, \cite{Bowditch}, \cite{Osinrelatively}, \cite{DrutuSapir}: groups admitting geometrically finite actions on proper geodesic Gromov hyperbolic spaces. These include fundamental groups of finite volume negatively curved manifolds and free products of arbitrary finite collections of groups.
\end{enumerate}

To better study these examples, one would like to associate to non-hyperbolic spaces a quasi-isometry invariant bordification analogous to the Gromov boundary.
One such object is the given by the Morse boundary, which roughly encodes all the hyperbolic directions.
More precisely, following Cordes \cite{Cordes}, given a function $N:[1,+\infty)\times [0,\infty) \to [0,\infty)$,
a geodesic $\alpha$ in a metric space $X$ is called $N$-Morse if for every point $x,y$ on $\alpha$, any $(\lambda,c)$ quasi-geodesic joining $x$ to $y$ stays within $N(\lambda,c)$ of $\alpha$.
The function $N$ is called a Morse gauge.

The $N$-Morse boundary of a metric space $X$,  $\partial_M^NX$, is the set of equivalence classes of Morse geodesic rays, where two $N$-Morse geodesic rays are declared to be equivalent if they stay within bounded distance of each other.
One can endow the $N$-Morse boundary with a topology, mimicking the definition of the topology on the Gromov boundary of a hyperbolic space $X$, that is, two equivalence classes of geodesic rays are close in this topology if they fellow travel for a long time.

The Morse boundary of $X$, $\partial_MX$, as a set, is the union of all the $N$-Morse boundaries.
For $\mathrm{CAT}(0)$ spaces, it coincides with the contracting boundary introduced by Charney and Sultan \cite{CharneySultan}.
It is topologized with the direct limit topology over all Morse gauges. The resulting topological space is a visibility space (i.e. every pair of points in the Morse boundary can be joined by a bi-infinite Morse geodesic) and it is invariant under quasi-isometry, see \cite{Cordes} for more details on all this.
Moreover, it is compact if and only if $X$ is hyperbolic, see \cite{Murray}.
However, this topology is not in general metrizable.
However, one can also endow the Morse boundary with a metrizable topology, called the Cashen--Mackay topology, see \cite{CashenMacKay}.
Whenever the Morse boundary embeds as a set into a metrizable set $Z$, such as the visual boundary of a CAT(0) space, it seems interesting to compare the Cashen--Mackay topology with the induced topology from $Z$, see \cite{Incerti-Medici} for instance.
We also refer to \cite{Cordessurvey} for many more properties of the Morse boundary.

\medskip
A major stream in geometric group theory is devoted to studying to what extent the algebraic and geometric properties of a group $G$ determine the properties of Markov chains on it, and conversely.
One way to do this is to relate asymptotic properties of random walks on a group to the dynamics of its action on some geometric boundary $Z$.
The goal of such an endeavor is often to show that typical paths in $G$ of the random walk generated by a probability measure $\mu$ converge to a point in the geometric boundary $Z$ and if possible, that $G\cup Z$ is in a measure theoretic sense the maximal bordification with this property.
This provides a geometric realization of (some quotient) of the Poisson boundary of the random walk and this identification can in turn provide geometric information, see \cite{KaimanovichVershik}, \cite{Erschlerannals}, \cite{Erschlersurvey}.

Our goal is to study the Morse boundary of Cayley graphs in this framework.
Unfortunately, for non-hyperbolic groups it is too small to be a model for the Poisson boundary: indeed as we will show, typical paths of the random walk do not converge to points in the Morse boundary.
Nevertheless, in this paper we connect geometry and probability in a different way, by showing the Morse boundary embeds into a probabilistically defined topological space called the Martin boundary.
The Martin boundary is defined as follows.
Consider a transient random walk on a finitely generated group $\Gamma$.
Let $F(g,h)$ be the probability that a random path starting at $g$ ever reaches $h$.
The expression $d_{G}(g,h)=-\log F(g,h)$ defines a (possibly asymmetric) metric on $\Gamma$ called the Green metric.
Its horofunction compactification is called the Martin boundary, see \cite{Sawyer}.
We will give more details on Poisson boundaries and Martin boundaries in Section~\ref{SectionPoissonMartin}.

Identifying the precise homeomorphism type of the Martin boundary is a difficult problem---in general different random walks on the same group can have wildly different Martin boundaries, see \cite{Gouezelannalsproba}.
It is therefore interesting to relate some geometric property of a group $\Gamma$ with Martin boundaries for large classes of random walks on $\Gamma$.
Ancona proved that for finitely supported random walks on hyperbolic groups, the Martin boundary is equivariantly homeomorphic to the Gromov boundary.
We prove a weaker result in a more general context.

\medskip
Recall that a relatively hyperbolic group is called non-elementary if its Bowditch boundary, which is the limit set into the Gromov boundary of a hyperbolic space $X$ on which the group geometrically finitely acts by isometries, is infinite.
Equivalently, the action on $X$ contains infinitely many independent loxodromic elements.

We also fix the following terminology for hierarchically hyperbolic groups.
Recall that such a group $\Gamma$ is equipped with an index set $\mathfrak{S}$ together with $\delta$-hyperbolic spaces $(CW,d_W)$, $W\in \mathfrak{S}$ (the constant $\delta$ is fixed).
It is also equipped with projection maps $\pi_W:\Gamma \to CW$.
Elements of $\mathfrak{S}$ are called domains.
The set of domains $\mathfrak{S}$ is endowed with a partial order with respect to which $\mathfrak{S}$ is either empty or has a unique maximal element $\mathbf{S}$.
The group $\Gamma$ acylindrically acts on on this maximal set $C\mathbf{S}$.
Whenever $C\mathbf{S}$ has unbounded diameter, this action is non-elementary and so in particular, $\Gamma$ is acylindrically hyperbolic.
We will say in this situation that $\Gamma$ is \textit{a non-elementary hierarchically hyperbolic group}.

\begin{theorem}\label{maintheoremMorsetoMartin}
Let $\Gamma$ be a non-elementary hierarchically hyperbolic group or a non-elementary relatively hyperbolic group whose parabolic subgroups have empty Morse boundary.
Let $\mu$ be a probability measure on $\Gamma$ whose finite support generates $\Gamma$ as a semi-group.
Then, the identity map on $\Gamma$  extends to an injective map $\Phi$ from the Morse boundary to the Martin boundary which is continuous with respect to the direct limit topology.
\end{theorem}

Indeed, in the proof of this theorem, we show that for each Morse gauge $N$, $\partial_M^N \Gamma$ topologically embeds in in the Martin boundary. Using some results from \cite{Cordes} we get a corollary that sheds some light on the topology of the Martin boundary of the mapping class group: For any $n \geq 2$ there exists a surface of finite type $S$ such that $\partial_\mu \mathrm{MCG}(S)$ contains a topologically embedded $(n-1)$-sphere, see Corollary~\ref{corollaryMCGs}.

Thus, the Morse boundary can be identified with a certain "canonical" subset of the Martin boundary.
For non-elementary relatively hyperbolic groups the corresponding result follows from a stronger result of \cite{GGPY}, but we provide another proof in this paper.
In Ancona's identification of the Martin boundary of a random walk on a hyperbolic group with the Gromov boundary, the main technical step is a certain deviation inequality asserting that the Green metric is roughly (up to an additive constant) additive along word geodesics.
Namely, if $g,h,w\in \Gamma$ are on a word geodesic aligned in this order, then $|d_{G}(g,w)-d_{G}(g,h)-d_{G}(g,w)|<C$ for a constant $C$ depending only $\Gamma$ and the random walk.
In order to prove Theorem~\ref{maintheoremMorsetoMartin} we show that a similar inequality holds along Morse geodesics.
\begin{theorem}\label{maintheoremMorseAncona}
Let $\Gamma$ be a non-elementary hierarchically hyperbolic group or a non-elementary relatively hyperbolic group whose parabolic subgroups have empty Morse boundary.
Let $\mu$ be a probability measure on $\Gamma$ whose finite support generates $\Gamma$ as a semi-group.
Then, for any Morse gauge $N$, there exists $C$ such that for any $N$-Morse geodesic $\alpha$ and for any points $g,h,w \in \alpha$ aligned in this order,
$$|d_{G}(g,w)-d_{G}(g,h)-d_{G}(g,w)|<C.$$
\end{theorem}
   
We do not know in general if the map constructed in Theorem \ref{maintheoremMorsetoMartin} is continuous for the Cashen-Mackay topology.
However, for relatively hyperbolic groups whose parabolic subgroups have empty Morse boundary, it is a homeomorphism on its image for this topology.
Indeed, \cite[Theorem~7.6]{CashenMacKay} shows that the Morse boundary endowed with this topology is embedded in the set of conical limit points in the Bowditch boundary, which in turns is embedded in the Martin boundary using results of \cite{GGPY}.
This leads to the following question.

\begin{question*}
Let $\Gamma$ be a non-elementary hierarchically hyperbolic group.
Is the map $\Phi$ a homeomorphism on its image for the Cashen--Mackay topology on the Morse boundary ?
\end{question*}

If this question has a negative answer, then we get a new metrizable topology on the Morse boundary, coming from the Martin boundary, which can be interesting in its own, at least from the perspective of random walks.

\medskip
We also investigate the connection between the Morse boundary and the Poisson boundary.
We prove that the image of the Morse boundary can be seen a Borel subset of the Martin boundary.
We can thus measure the Morse boundary with respect to the harmonic measure $\nu$.
We then prove the following.

\begin{theorem}\label{maintheoremmeasureMorse}
Let $\Gamma$ be a non-elementary hierarchically hyperbolic group or a non-elementary relatively hyperbolic group whose parabolic subgroups have empty Morse boundary.
Let $\mu$ be a probability measure on $\Gamma$ whose finite support generates $\Gamma$ as a semi-group and let $\nu$ be the corresponding harmonic measure on the Martin boundary.
Then $\nu(\partial_M\Gamma)=0$ unless $\Gamma$ is hyperbolic, in which case $\nu(\partial_M\Gamma)=1$.
\end{theorem}

Using results of Maher and Tiozzo \cite{MaherTiozzo}, one could prove
that the Morse boundary can be embedded in another realization of the
Poisson boundary, namely the Gromov boundary of a space on which the
group acyindrically acts. We want to emphase that the Poisson boundary
is a measure theoretical object and that we could not a priori deduce from it
that there exists a continuous map from the Morse boundary to the Martin
boundary.

\medskip
Let us also mention the following.
An acylindrically hyperbolic group $\Gamma$ may admit various non-elementary acylindrical actions on hyperbolic spaces.
A fruitful line of research initiated by Abbott \cite{Abbott} is to find the best possible one.
Such an action $\Gamma \curvearrowright X$ is called universal if for any element $g$ of $\Gamma$ such that there exists an acylindrical action of $\Gamma$ for which $g$ is loxodromic, the action of $g$ on $X$ also is loxodromic.
One can also introduce a partial order on cobounded acylindrical actions, see \cite{AbbottBalasubramanyaOsin}.
When existing, a maximal action for this partial order is called a largest acylindrical action.
Any largest action is necessarily a universal action and is unique.
Abbott, Behrstock and Durham \cite{AbbottBehrstockDurham} proved that any non-elementary hierarchically hyperbolic group admits a largest acylindrical action.
More precisely, they proposed a way to modify the hierarchical structure of the group so that the action of $\Gamma$ on $C\mathbf{S}$ is a largest acylindrical action, where $\mathbf{S}\in \mathfrak{S}$ is the maximal domain.
We will use this modified hierarchical structure in the following, see in particular the discussion after Proposition~\ref{propMorselinearprogressrelativelyhyperbolic}.

\subsection*{Organization of the paper}
In Section~\ref{SectionPoissonMartin}, we recall the precise definition of the Poisson boundary and the Martin boundary  and review known results about their identifications with geometric boundaries.

Section~\ref{Sectiondeviation} is devoted to the proof of Theorem~\ref{maintheoremMorseAncona}.
We first prove an enhanced version of deviation inequalities in acylindrically hyperbolic groups obtained in \cite{MathieuSisto}.
These inequalities basically state the conclusions of Theorem~\ref{maintheoremMorseAncona} hold, provided that $x,y,z$ are well aligned in the hyperbolic space $X$ on which the group acylindrically acts.
We then show that this condition is satisfied for any points $x,y,z$ on a Morse geodesic in a hierarchically hyperbolic group.

In Section~\ref{Sectionconstructionmap}, we use these inequalities to construct the map from the Morse boundary to the Martin boundary and we prove Theorem~\ref{maintheoremMorsetoMartin}.
The construction, adapted from \cite{KaimanovichErgodicity}, uses a bit of potential theory.
Once the map is constructed, injectivity is proved exactly like in hyperbolic groups.
On the other hand, the proof of continuity is new and different, since the constant $C$ in Theorem~\ref{maintheoremMorseAncona} depends on the Morse gauge, while it is fixed for hyperbolic groups.

Finally, in Section~\ref{Sectionmeasurezero}, we prove Theorem~\ref{maintheoremmeasureMorse}.
We actually give two proofs.
The first one basically only uses ergodicity of the harmonic measure and
it seems it could be adapted other contexts.
However, it only works for symmetric random walks, so
we give a second proof which is a bit more specific but which does not need such an assumption.

\subsection*{Acknowledgements} The authors would like to thank the organizers of Young Geometric Group Theory VII in Les Diablerets, Switzerland and the 3-manifolds and Geometric Group Theory conference in Luminy, France where part of this work was accomplished. The first author was supported by the ETH Zurich Postdoctoral Fellowship Program, cofunded by a Marie Curie Actions for People COFUND Program. The first author was also supported at the Technion by a Zuckerman STEM Leadership Fellowship and Israel Science Foundation (Grant 1026/15). The third author was partially supported by the National Science and Engineering Research Council of Canada (NSERC).

\section{Random walks and probabilistic boundaries}\label{SectionPoissonMartin}
\subsection{The Poisson boundary and the Martin boundary}
Consider a finitely generated group $\Gamma$ and a probability measure $\mu$ on $\Gamma$.
The random walk driven by $\mu$ is defined as $X_n=g_1...g_n$, where $g_k$ are independent random variables following the law of $\mu$.
We consider two probabilistic boundaries in this paper, the Poisson boundary and the Martin boundary.

\medskip
The Poisson boundary of a group $\Gamma$ endowed with a probability measure $\mu$ is the space of ergodic components for the time shift in the path-space of the associated random walk \cite{Erschlersurvey}, \cite{KaimanovichVershik}.
It is also isomorphic to a maximal measurable space endowed with a stationary probability measure $\lambda$ such that the random walk almost surely converges in the measure theoretical sense to a point in the boundary, that is, $X_n\cdot \lambda$ almost surely converges to a Dirac measure, see \cite{KaimanovichPoissonhyperbolic} and \cite{Furstenberg73}.
We emphasize that the Poisson boundary is a purely measure theoretical space, unlike the topological Martin boundary, which we now define.

\medskip
We introduced the Green metric in the introduction.
Let us give more details now.
The Green function associated with $\mu$ is defined as
$$G(g,h)=\sum_{n\geq 0}\mu^{*n}(g^{-1}h),$$
where $\mu^{*n}$ is the $n$th power of convolution of $\mu$.
Let $F(g,h)=\frac{G(g,h)}{G(e,e)}$.
Then, $F(g,h)$ is the probability of ever reaching $h$, starting the random walk at $g$, see for example \cite[Lemma~1.13~(b)]{Woess}.
The Green metric $d_G$ is then defined as
$$d_G(g,h)=-\log F(g,h).$$
When the measure $\mu$ is symmetric, this is indeed a distance, we refer to \cite{BlachereBrofferio} for more details, where this metric was first introduced.
The triangle inequality can be reformulated as
\begin{equation}\label{triangleGreen}
F(g_1,g_2)F(g_2,g_3)\leq F(g_1,g_3).
\end{equation}
In particular, for any $g_1,g_2,g_3$, we have
\begin{equation}\label{triangleGreen'}
G(g_1,g_2)G(g_2,g_3)\leq CG(g_1,g_3)
\end{equation}
for some uniform constant $C$.
Note that this inequality is always true, whether $\mu$ is symmetric or not, since it only states that the probability of reaching $g_3$ starting at $g_1$ is always bigger than the probability of first reaching $g_2$ from $g_1$, then $g_3$ from $g_2$.
The Martin boundary is then the horofunction boundary associated with the Green metric $d_G$ on $\Gamma$.
More precisely, introduce the Martin kernel $K(\cdot,\cdot)$ as
$$K(g,h)=\frac{G(g,h)}{G(e,h)}.$$
Then, the Martin compactification is a topological space $\overline{\Gamma}^{\mu}$ such that
\begin{enumerate}
\item the space $\Gamma$ endowed with the discrete topology is a dense and open space in $\overline{\Gamma}^{\mu}$,
\item letting $\partial_\mu\Gamma=\overline{\Gamma}^{\mu}\setminus \Gamma$, a sequence $g_n$ of elements of $\Gamma$ converges to a point in $\overline{\Gamma}^{\mu}$ if and only if $K(\cdot,g_n)$ converges pointwise to a function. If $\xi$ is the corresponding limit in $\partial_{\mu}\Gamma$, we will write $K_{\xi}$ for the corresponding limit function.
\end{enumerate}
The complement $\partial_\mu\Gamma$ of $\Gamma$ in the Martin compactification $\overline{\Gamma}^{\mu}$ is called the Martin boundary.
Both the Martin compactification and the Martin boundary are unique up to homeomorphism.
Moreover, they both are metrizable spaces.
This definition makes sense whether $\mu$ is symmetric or not and whether $d_G$ is a true distance or not.
We refer to \cite{Sawyer} for a detailed construction.

\medskip
One important aspect of the Poisson boundary and the Martin boundary is their connection with harmonic functions.
Recall that a function $f:\Gamma \to \R$ is called harmonic (with respect to $\mu$) if for every $g\in \Gamma$,
$$f(g)=\sum_{h\in \Gamma}\mu(g^{-1}h)f(h).$$
The following key theorem states that every positive harmonic function can be represented as an integral on the Martin boundary.
\begin{theorem*}[Martin representation theorem]\cite[Theorem~4.1]{Sawyer}
Let $\Gamma$ be a finitely generated group with a probability measure $\mu$ and assume that the random walk driven by $\mu$ is transient.
For every positive harmonic function $f$ on $(\Gamma,\mu)$, there exists a borelian measure $\nu_f$ on $\partial_\mu\Gamma$ such that for every $g\in \Gamma$,
$$f(g)=\int K_\xi(g)d\nu_f(\xi).$$
\end{theorem*}

In general, the measure $\nu_f$ is not unique.
We restrict ourselves to the minimal boundary to obtain uniqueness.
A positive harmonic function $f$ on $(\Gamma,\mu)$ is called minimal if for every positive harmonic function $\tilde{f}$ satisfying $\tilde{f}\leq Cf$ for some constant $C$, $\tilde{f}=C'f$ for some constant $C'$.
The minimal boundary is defined as
$$\partial_\mu^{\min}\Gamma=\left \{\xi \in \partial_\mu\Gamma, K_\xi \text{ is harmonic and minimal}\right \}.$$
Then, for every positive harmonic function $f$, one can choose the measure $\nu_f$ giving full measure to $\partial_\mu^{\min}\Gamma$ and in this case, $\nu_f$ is unique, see \cite{Sawyer} for more details.

The main relation between the Poisson boundary and the Martin boundary is as follows.
\begin{theorem*}
Let $\Gamma$ be a finitely generated group with a probability measure $\mu$ and assume that the random walk $X_n$ driven by $\mu$ is transient.
Then, the random walk $X_n$ almost surely converges to a point in the Martin boundary.
Letting $X_{\infty}$ be the corresponding limit,
denote by $\nu$ the law of $X_{\infty}$ on $\partial_\mu\Gamma$.
Then, $(\partial_\mu\Gamma,\nu)$ is isomorphic as a measured space to the Poisson boundary.
\end{theorem*}

For more details on the connection between the two boundaries, we refer to the survey \cite{Kaimanovichsurvey}, and in particular to \cite[Section~2.2]{Kaimanovichsurvey}.

\subsection{Comparing the boundaries}
Trying to identify the Poisson or the Martin boundary with a geometric boundary has been a fruitful line of research, initiated by the work of Furstenberg \cite{Furstenberg63}, \cite{Furstenberg73}.
A landmark result in this direction is the Kaimanovich criterion, stated in \cite{KaimanovichPoissonhyperbolic} which allows one to identify the Poisson boundary for some classes of measures $\mu$, see \cite[Theorem~6.4]{KaimanovichPoissonhyperbolic} for more details.

This criterion applies in many situations.
For any hyperbolic group $\Gamma$ and any probability measure $\mu$, whose support generates $\Gamma$ as a semi-group, the random walk almost surely converges to a point in the Gromov boundary of $\Gamma$.
When $\mu$ has finite entropy and finite logarithmic first moment, one can use the Kaimanovich criterion to prove that the Gromov boundary endowed with the corresponding limit measure is a model for the Poisson boundary.
More generally, Maher and Tiozzo \cite{MaherTiozzo} proved that for any group $\Gamma$ acting on a hyperbolic space $X$, for any non-elementary probability measure $\mu$ on $\Gamma$, the image of the random walk in $X$ almost surely converges to a point in the Gromov boundary $\partial X$ of $X$.
Moreover, if the action is acylindrical, they used the Kaimanovich criterion to prove that whenever $\mu$ has finite entropy and finite logarithmic first moment, $\partial X$ endowed with the corresponding limit measure is a model for the Poisson boundary.
In particular, for all the groups we consider in this paper, the Poisson boundary is identified with a geometric boundary, namely the Gromov boundary on any hyperbolic space $X$ the group acylindrically acts on.

\medskip
On the other hand, as explained in the introduction, identifying the Martin boundary is a much more difficult task.
Ancona proved in \cite{AnconaMartinhyperbolic} that for every hyperbolic group $\Gamma$ and every probability measure $\mu$ whose finite support generates $\Gamma$ as a semi-group, the Martin boundary is homeomorphic to the Gromov boundary.

Recently, Gekhtman, Gerasimov, Potyagailo and Yang proved that for finitely supported measures on a relatively hyperbolic group, the Martin boundary always covers the Bowditch boundary \cite{GGPY}.
The preimage of a conical limit point is always reduced to a point and actually, conical limit points embed into the Martin boundary.
Note that whenever the parabolic subgroups have empty Morse boundary, the Morse boundary of the group can be seen as a subset of conical limit points, so results of \cite{GGPY} provide a more direct proof of our result in this context.

It is expected that the Martin boundary is bigger than the Bowditch boundary and that the preimage of a parabolic limit point is the Martin boundary of the induced walk on the corresponding parabolic subgroup.
This is proved for hyperbolic groups with respect to virtually abelian groups in \cite{DGGP}.

To the authors' knowledge, not much is known in general about the Martin boundary of a hierarchically hyperbolic group, even for mapping class groups.
Note however that Kaimanovich and Masur \cite{KaimanovichMasur} used the Kaimanovich criterion to show that Thurston's PMF boundary of Teichm\"uller space is a model for the Poisson boundary of the mapping class group, and the stationary measure therein gives full weight to endpoints inside the PMF boundary of Teichm\"uller geodesics recurring to a fixed subset of Teichm\"uleer space.
It seems reasonable to conjecture that this set of recurrent foliations (which may be considered the direct analogue of conical limit points for relatively hyperbolic groups) can be embedded into the Martin boundary, which would provide a direct proof that the Morse boundary embeds into the Martin boundary in this context.
Our result can be viewed as a small step in this direction.

\section{Deviation inequalities}\label{Sectiondeviation}
\subsection{Global-Ancona inequalities in acylindrically hyperbolic groups}
Consider a finitely generated group $\Gamma$ acting acylindrically on a hyperbolic space $X$.
If $\mathcal{S}$ is a finite generating set for $\Gamma$, write $d_{\mathcal{S}}$ for the word distance associated with $\mathcal{S}$ or simply $d$ whenever $\mathcal{S}$ is fixed.
Also, given the choice of a fixed point $o\in X$,
write $d_X$ for the induced distance in $\Gamma$, that is
$d_X(g,h)=d_X(g\cdot o,h\cdot o)$.
A finite sequence $\alpha=g_1,...,g_n$ of points in $\Gamma$ is called a path if $g^{-1}_{i}g_{i+1}\in  \mathcal{S}$ for each $i$. It's length in the word metric induced by $(\Gamma, S)$ will be denoted $l_{\Gamma}(\alpha)$.

Recall the following definition from \cite{MathieuSisto}.
\begin{definition}
Fix a finite generating set $\mathcal{S}$ for $\Gamma$ and let $g,h\in \Gamma$ and $\alpha$ be a $d_\mathcal{S}$ word geodesic from $g$ to $h$.
Let $T,S\geq 1$.
Then, a point $p$ on $\alpha$ is called a $(T,S)$-linear progress point if for every $p_1,p_2$ on $\alpha$ such that $p_1,p,p_2$ are aligned in this order and such that $d(p,p_1)\geq S$, $d(p,p_2)\geq S$, we have
$$d(p_1,p_2)\leq Td_X(p_1,p_2).$$
\end{definition}

Whenever $f$ and $g$ are two functions such that there exists $C$ such that $\frac{1}{C}f\leq g \leq C f$, we write $f\asymp g$.
When the implied constant $C$ depends on some parameters, we will avoid this notation, except if the dependency is clear from the context.
Also, whenever there exists $C$ such that $f\leq C g$, we will write $f\lesssim g$.
Ancona inequalities were stated using the Green metric in the introduction, but notice that one can reformulate them as follows.
If $\Gamma$ is Gromov hyperbolic, then for any $x,y,z$ aligned on a geodesic, we have
\begin{equation}\label{Ancona}
G(x,z)\asymp G(x,y)G(y,z).
\end{equation}
Our goal in this section is to prove Ancona-type inequalities for linear progress points on a word-geodesic.

We first introduce some notations.
We consider a trajectory for the random walk $\beta=(\beta_0,...,\beta_n)$ of length $n$.
We write
$$W(\beta)=\mu(\beta_0^{-1}\beta_1)...\mu(\beta_{n-1}^{-1}\beta_n)$$
and we call $W(\beta)$ the weight of the trajectory $\beta$.
Also, given a collection $\mathcal{T}$ of trajectories for the random walk, we write
$$W(\mathcal{T})=\sum_{\beta \in \mathcal{T}}W(\beta).$$
In particular, letting $\mathcal{T}_n(g,h)$ be the collection of trajectories of length $n$ from $g$ to $h$, we can rewrite the Green function from $g$ to $h$ as
$$G(g,h)=\sum_{n\geq 0}W(\mathcal{T}_n(g,h)).$$

Also, for a subset $A$ of $\Gamma$, we denote by $G(g,h;A)$ the contribution to $G(g,h)$ of trajectories all of whose points lie in $A$, with the possible exception of the endpoints.
In other words,
$$G(g,h;A)=\delta_{g,h}+\sum_{n\geq 1}\sum_{g_1,...,g_{n-1}\in A}\mu(g^{-1}g_1)\mu(g_1^{-1}g_2)...\mu(g_{n-1}^{-1}h),$$
where $\delta_{g,h}=0$ if $g\neq h$ and $\delta_{g,h}=1$ if $g=h$.


\begin{proposition}\label{epsilonAnconaAcylindricallyhyperbolic}
Let $\Gamma$ be a finitely generated acting acylindrically on a hyperbolic space $X$.
For every $T,S\geq 1$ and every $\epsilon>0$ , there exists $C_1\geq 0$ such that the following holds.
Let $g,h\in \Gamma$ and let $\alpha$ be a geodesic connecting $g$ to $h$.
Let $p$ be a $(T,S)$-linear progress point on $\alpha$.
Then,
$$G(g,h;B_{C_1S}(p)^c)\leq \epsilon G(g,h).$$
\end{proposition}

We will need the following geometric lemma.
\begin{lemma}\cite[Proposition~10.4]{MathieuSisto}\label{prop10.4MathieuSisto}
Let $\Gamma$ be a finitely generated group with a fixed finite generating set. 
Assume that $\Gamma$ acts acylindrically on a hyperbolic space $X$.
For any $L$, there exist a constant $C$ and a diverging function $\rho:\mathbb{R}_+\to \mathbb{R}_+$ such that the following holds.
Let $\alpha_1$ and $\alpha_2$ be two $L$-Lipschitz paths in the Cayley graph of $\Gamma$.
Write $g_1,h_1$, respectively $g_2,h_2$ the endpoints of $\alpha_1$, respectively $\alpha_2$.
Then,
$$\max \{l_{\Gamma}(\alpha_1),l_{\Gamma}(\alpha_2)\}\geq \bigg (d_X(g_1,h_1)-d_X(g_1,g_2)-d_X(h_1,h_2)-C\bigg )\rho(d_{\Gamma}(\alpha_1,\alpha_2)).$$
\end{lemma}

We can now prove Proposition~\ref{epsilonAnconaAcylindricallyhyperbolic}, which is a refinement of \cite[Lemma~12.3]{MathieuSisto}.
Looking carefully, one can see that our statement is actually proven there.
We still rewrite the complete proof for convenience.
\begin{proof}
We first introduce some notations.
Recall that the support of $\mu$ is finite, so there exists $0<q<1$ such that for any trajectory $\beta$ of the random walk of length $n$, we have $W(\beta)\geq q^n$.
Also recall that the support of $\mu$ generates $\Gamma$ as a semi-group.
In particular, there exists $\Lambda$ such that a trajectory for the random walk of minimal length connecting points at distance $l$ has length at most $\Lambda l$.
Since $\Gamma$ is non-amenable, the spectral radius of the random walk, which is the radius of convergence of the Green function is bigger than 1, see \cite{Kesten}.
Hence, there exists $\theta<1$ such that for any $g\in G$, for any $n$, $\mu^{*n}(g)\leq \theta^n$.
In particular, given $g,h\in \Gamma$,
$$G(g,h)=\sum_{n\geq 0}\mu^{*n}(g^{-1}h)=\sum_{n\geq \Lambda d(g,h)}\mu^{*n}(g^{-1}h)\leq K\theta^{\Lambda d(g,h)}$$
for some $K$.
More generally, if $g$ and $h$ are fixed and if $\beta$ is a collection of trajectories for the random walk from $g$ to $h$ of length at least $n_0$, we have
\begin{equation}\label{weightofalongpath}
W(\beta)\leq K\theta^{n_0}.
\end{equation}
Finally, since $\Gamma$ has at most exponential growth, there exists $v$ such that for any $R$
the cardinality of a ball of radius $R$ is at most $\mathrm{e}^{vR}$.

We will need to use three constants $N$, $C_0$ and $C_1$.
To make it easier to understand, we explain now how we choose these constants.
We will choose $N$ large enough, that will only depend on the random walk (precisely, it will depend on $\theta$ and on $\Lambda$).
We will then choose $C_0$ that will depend on $N$ and $T$.
Finally we will choose $C_1$ that will depend on $N$, $C_0$, $S$, $T$ and $\epsilon$.

Precisely, we choose $N$ such that
$\theta^N/q^\Lambda \leq 1/2$.
We then choose $C_0$ such that
$\rho(t)\geq 2NT$ for every $t\geq C_0$, where the function $\rho$ is given by Lemma~\ref{prop10.4MathieuSisto}.
We also assume that $C_0\geq 2\Lambda$.
Finally, we choose $C_1$ such that the following conditions hold.
First, $C_1\geq 3C_0$ and then for every $t\geq 2C_1S-4C_0S$, we have
\begin{enumerate}[(a)]
\item $t-2C_0 \geq \frac{t}{2}+T(2C_0S+C)$, where $C$ is the constant given by Lemma~\ref{prop10.4MathieuSisto},
\item $t+2C_0S\leq 2t$,
\item $q^{-4\Lambda C_0S}\mathrm{e}^{2v\Lambda C_0S}2^{-t}\leq \frac{\epsilon}{K}$, where $K$ is the constant in~(\ref{weightofalongpath}).
\end{enumerate}

Let $g,h,p$ be as in the statement of the proposition.
Given any trajectory $\beta$ for the random walk from $g$ to $h$ that avoids a large ball around $p$, we first construct a long sub-trajectory $\gamma$ as follows.
Let $\beta=(w_0,...,w_n)$, $w_0=g$ and $w_n=h$ and assume that $\beta$ avoids the ball of radius $C_1S$ centered around $p$.
Consider the last point $g'$ on the trajectory $\beta$ which is within a distance at most $C_0S$ from a point $g''$ on the geodesic $\alpha$ and such that $g$, $g''$ and $p$ are aligned in this order.
Similarly, consider the first point $h'$ on $\beta$ after $g$' which is within a distance at most $C_0S$ from a point $h''$ on $\alpha$, where $p$, $h''$ and $h$ are aligned in this order.

Let $\gamma$ be the sub-trajectory of $\beta$ starting at $g'$ and ending at $h'$.
Notice that $g'\neq h'$.
Indeed, a point $g_0$ on $\beta$ cannot be simultaneously $C_0S$-close to points $g_1$, respectively $g_2$, on the geodesic $\alpha$ such that $g,g_1,p$, respectively $p,g_2,h$ are aligned in this order.
Otherwise, one would have
$$d(g_0,p)\leq d(g_0,g_1) +d(g_1,p) \leq C_0S+ d(g_1,g_2)\leq 3C_0S<C_1S,$$
which is a contradiction since $\beta$ stays outside $B_{C_1S}(p)$.
Also, $\gamma$ only intersects the $C_0S$-neighborhood of the geodesic $\alpha$ at $g'$ and $h'$.
Indeed if one point $g_0$ of $\gamma$ was $C_0S$-close to $\alpha$, then there would be a point $g_1$ on $\alpha$ with $d(g_0,g_1)\leq C_0S$.
If $g,g_1,p$ were aligned in this order, this would contradict the definition of $g'$. If not, this would contradict the definition of $h'$.

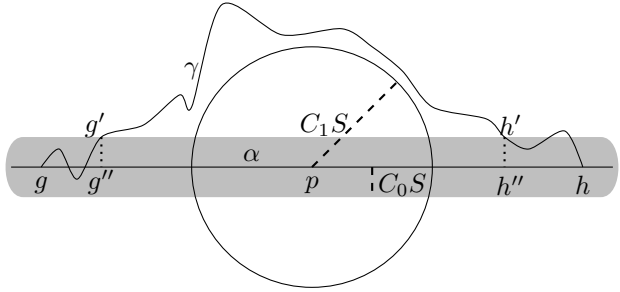
\begin{center}
\begin{tikzpicture}[scale=.8]
\fill[gray!50] (-4.8,-.5) to[out=180, in=180] (-4.8,.5) -- (4.8,.5) to[out=0, in=0] (4.8,-.5) -- (-4.8,-.5) -- cycle ;
\draw (-5,0)--(5,0) ;
\draw (2,0) arc(0:360:2) ;
\draw[dashed, thick] (0,0)--(1.414,1.414) ;
\draw[dashed, thick] (1,0)--(1,-.5) ;
\draw[dotted, thick] (-3.5,.5)--(-3.5,0) ;
\draw[dotted, thick] (3.2,.5)--(3.2,0) ;
\draw plot[smooth] coordinates {(-4.5,0) (-4.2,.3) (-3.9,-.2) (-3.5,.5) (-2.8,.7) (-2.2,1.2) (-2,1) (-1.5,2.7) (-.5,2.2) (.5,2.3) (1,2) (1.5,1.6) (2,1) (2.9,.8) (3.2,.5) (3.6,.3) (4.2,.6) (4.5,0)} ;
\draw (-4.5,-.3) node{$g$} ;
\draw (4.5,-.3) node{$h$} ;
\draw (-1,.2) node{$\alpha$} ;
\draw (0,-.3) node{$p$} ;
\draw (-3.6,.7) node{$g'$} ;
\draw (3.3,.7) node{$h'$} ;
\draw (-3.5,-.25) node{$g''$} ;
\draw (3.3,-.3) node{$h''$} ;
\draw (-2,1.6) node{$\gamma$} ;
\draw (1.5,-.3) node{$C_0S$} ;
\draw (.2,.7) node{$C_1S$} ;
\end{tikzpicture}
\end{center}

Denote by $l(\gamma)$ the length of $\gamma$ in $\Gamma$.
Then, Lemma~\ref{prop10.4MathieuSisto} shows that
\begin{equation}\label{lengthgamma}
\mathrm{max}(l(\gamma),d(g'',h''))\geq (d_X(g'',h'')-d_X(g',g'')-d_X(h',h'')-C)\rho(C_0S).
\end{equation}

Note that $d(g',h')\geq d(g'',h'')-2C_0S$ and since $g'',p,h''$ are aligned in this order,
$$d(g',h')\geq d(g'',p)+d(p,h'')-2C_0S\geq d(g',p)+d(h',p)-4C_0S$$
so finally
$$d(g',h')\geq 2C_1S-4C_0S.$$
In particular, according to Condition~(a) above,
$$\frac{d(g',h')-2C_0}{T} \geq \frac{d(g',h')}{2T}+2C_0S+C.$$
Since $p$ is a $(T,S)$-linear progress point,
$$d_X(g'',h'')\geq \frac{d(g'',h'')}{T}\geq \frac{d(g',h')-2C_0S}{T}\geq \frac{d(g',h')}{2T}+2C_0S+C.$$
Since $d_X(\cdot,\cdot)\leq d(\cdot,\cdot)$ and $d(g',g'')\leq C_0S$, $d(h',h'')\leq C_0S$,~(\ref{lengthgamma}) yields
$$\mathrm{max}(l(\gamma),d(g'',h''))\geq\frac{d(g',h')}{2T}\rho(C_0S).$$
According to the condition defining $C_0$, we get
$$\mathrm{max}(l(\gamma),d(g'',h''))\geq Nd(g',h').$$
Finally, notice that $d(g'',h'')\leq d(g',h')+2C_0S$ so Condition~(b) above shows that
$d(g'',h'')\leq 2d(g',h')$ and so $d(g'',h'')<Nd(g',h')$.
We thus get
\begin{equation}\label{lengthgamma2}
l(\gamma)\geq Nd(g',h')\geq N(2C_1S-4C_0S).
\end{equation}

We now construct a trajectory $\hat{\beta}$ for the random walk that will replace $\beta$ as follows.
Consider a trajectory $\hat{\gamma}_1$ of minimal length from $g'$ to the geodesic $\alpha$ and denote by $\hat{g}'$ the endpoint of this trajectory on $\alpha$.
Similarly, consider a trajectory $\hat{\gamma}_2$ of minimal length from $\alpha$ to $h'$ with initial point $\hat{h}'$.
Note that one can find a trajectory from $g'$ to $g''$ of length at most $\Lambda C_0S$ and similarly with $h'$ and $h''$, so that in particular
$d(g',\hat{g}')\leq \Lambda C_0S$ and
$d(h',\hat{h}')\leq \Lambda C_0S$.
Also, $\hat{g}'$, $p$ and $\hat{h}'$ are aligned in this order.
Now, consider trajectories of minimal length connecting successive points on the sub-geodesic of $\alpha$ from $\hat{g'}$ to $\hat{h}'$ and denote by $\hat{\gamma}_3$ the concatenation of these trajectories.
Denote by $\hat{\gamma}$ the concatenation of $\hat{\gamma}_1$, $\hat{\gamma}_3$ and $\hat{\gamma}_2$.
Then, $\hat{\gamma}$ is a trajectory for the random walk starting at $g'$ that joins the geodesic $\alpha$, roughly follows it and then goes to $h'$.
Notice that $\hat{\gamma}$ stays in the $\Lambda$-neighborhood of $\alpha$.
Finally, let $\hat{\beta}$ be the concatenation of the sub-trajectory $\hat{\beta}_1$ from $g$ to $g'$, the trajectory $\hat{\gamma}$ and the sub-trajectory $\hat{\beta}_2$ from $h'$ to $h$.

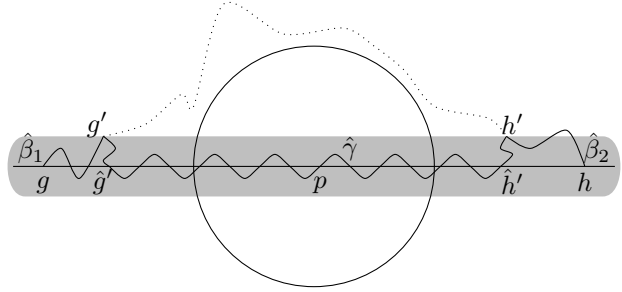
\begin{center}
\begin{tikzpicture}[scale=.8]
\fill[gray!50] (-4.8,-.5) to[out=180, in=180] (-4.8,.5) -- (4.8,.5) to[out=0, in=0] (4.8,-.5) -- (-4.8,-.5) -- cycle ;
\draw (-5,0)--(5,0) ;
\draw (2,0) arc(0:360:2) ;
\draw plot[smooth] coordinates {(-4.5,0) (-4.2,.3) (-3.9,-.2) (-3.5,.5)} ;
\draw[dotted] plot[smooth] coordinates {(-3.5,.5) (-2.8,.7) (-2.2,1.2) (-2,1)  (-1.5,2.7) (-.5,2.2) (.5,2.3) (1,2) (1.5,1.6) (2,1) (2.9,.8) (3.2,.5)} ;
\draw plot[smooth] coordinates {(3.2,.5) (3.6,.3) (4.2,.6) (4.5,0)} ;
\draw plot[smooth] coordinates {(-3.5,.5) (-3.3,.3) (-3.5,.1) (-3.4,0)} ;
\draw plot[smooth] coordinates {(-3.4,0) (-3.15,-.2) (-2.9,0) (-2.65,.2) (-2.4,0) (-2.15,-.2) (-1.9,0) (-1.65,.2) (-1.4,0) (-1.15,-.2) (-.9,0) (-.65,.2) (-.4,0) (-.15,-.2) (.1,0) (.35,.2) (.6,0) (.85,-.2) (1.1,0) (1.35,.2) (1.6,0) (1.85,-.2) (2.1,0) (2.35,.2) (2.6,0) (2.85,-.2) (3.1,0)} ;
\draw plot[smooth] coordinates {(3.1,0) (3.3,.2) (3.1,.3) (3.2,.5)} ;
\draw (-4.5,-.3) node{$g$} ;
\draw (4.5,-.25) node{$h$} ;
\draw (.1,-.3) node{$p$} ;
\draw (-3.6,.7) node{$g'$} ;
\draw (3.3,.7) node{$h'$} ;
\draw (-3.5,-.25) node{$\hat{g}'$} ;
\draw (3.3,-.25) node{$\hat{h}'$} ;
\draw (-4.7,.3) node{$\hat{\beta}_1$} ;
\draw (.6,.3) node{$\hat{\gamma}$} ;
\draw (4.7,.3) node{$\hat{\beta}_2$} ;
\end{tikzpicture}
\end{center}

This construction defines a map $\Psi:\beta\mapsto \hat{\beta}$.
To conclude, we just need to see that the total weight of the preimage of a trajectory $\hat{\beta}$ under this map is small, compared to the weight of $\hat{\beta}$.

Precisely, let $\hat{\beta}$ be a trajectory for the random walk constructed as above and let $\beta$ be such that $\Psi(\beta)=\hat{\beta}$.
Then, $\beta$ is the concatenation of the sub-trajectory $\hat{\beta}_1$ of $\hat{\beta}$ from $g$ to some $g'$, a trajectory that avoids $B_{C1S}(p)$ from $g'$ to some $h'$ and the sub-trajectory $\hat{\beta}_2$ of $\hat{\beta}$ from $h'$ to $h$.
Moreover, $\hat{g}'$ and $\hat{h}'$ are completely determined by $\hat{\beta}$.
Indeed, since $C_0S\geq 2\Lambda$, these points coincide with the last points on $\hat{\beta}$ around $p$ that intersect $\alpha$ before the trajectory leaves the $2\Lambda$-neighborhood of $\alpha$.
Also, recall that $d(g',\hat{g}')\leq \Lambda C_0S$ and $d(h',\hat{h}')\leq \Lambda C_0S$.
Finally, once $g'$ and $h'$ are fixed, $\hat{\beta}_1$ and $\hat{\beta}_2$ are completely determined.
Letting $\mathcal{T}_{\hat{\beta}}$ be the collection preimages of $\hat{\beta}$, this proves that
$$W(\mathcal{T}_{\hat{\beta}})\leq \sum_{g'\in B_{\Lambda C_0S}(\hat{g}')}\sum_{h'\in C_{\Lambda C_0S}(\hat{h}')}W(\hat{\beta}_1)G(g',h';B_{C_1S}(p)^c)W(\hat{\beta}_2).$$
According to~(\ref{lengthgamma2}), the length of the subtrajectory $\gamma$ of $\beta$ from $g'$ to $h'$ is at least $Nd(g',h')$,
so that~(\ref{weightofalongpath}) shows that
$$G(g',h';B_{C_1S}(p)^c)\leq K\theta^{Nd(g',h')}.$$
On the other hand, the sub-trajectory $\hat{\gamma}$ of $\hat{\beta}$ from $g'$ to $h'$ has length at most $\Lambda d(g',h')+4\Lambda C_0S$, so that
$$W(\hat{\beta})\geq  W(\hat{\beta}_1)W(\hat{\beta}_2)q^{\Lambda d(g',h')+4\Lambda C_0S}$$
and so
$$W(\mathcal{T}_{\hat{\beta}})\leq\sum_{g'\in B_{\Lambda C_0S}(\hat{g}')}\sum_{h'\in B_{\Lambda C_0S}(\hat{h}')} W(\hat{\beta})K\theta^{Nd(g',h')} q^{-\Lambda d(g',h')-4\Lambda C_0S}.$$
According to the condition defining $N$,
$$\theta^{Nd(g',h')} q^{-\Lambda d(g',h')-4\Lambda C_0S}\leq 2^{-d(g',h')}q^{-4\Lambda C_0S}$$
and since $d(g',h')\geq 2C_1S-4C_0S$, Condition~(c) above shows that
$$2^{-d(g',h')}\leq q^{4\Lambda C_0S}\mathrm{e}^{-2v\Lambda C_0S}\frac{\epsilon}{K}.$$
Since the balls $B_{\Lambda C_0S}(\hat{g}')$ and $B_{\Lambda C_0S}(\hat{h}')$ have cardinality at most $\mathrm{e}^{v\Lambda C_0S}$, we get
$$W(\mathcal{T}_{\hat{\beta}})\leq W(\hat{\beta})\mathrm{e}^{2v\Lambda C_0S}Kq^{-4\Lambda C_0S}q^{4\Lambda C_0S}\mathrm{e}^{-2v\Lambda C_0S}\frac{\epsilon}{K}=\epsilon W(\hat{\beta}).$$
This concludes the proof.
\end{proof}

\begin{corollary}\label{corollaryAnconalinearprogress}
Let $\Gamma$ be a finitely generated acting acylindrically on a hyperbolic space $X$.
For every $T,S\geq 1$ there exists $C\geq 1$ such that the following holds.
Let $g,h\in \Gamma$ and let $\alpha$ be a geodesic connecting $g$ to $h$.
Let $p$ be a $(T,S)$-linear progress point on $\alpha$.
Then,
$$\frac{1}{C}G(g,p)G(p,h)\leq G(g,h)\leq C G(g,p)G(p,h).$$
\end{corollary}

\begin{proof}
Let $T,S\geq 1$ and let $p$ be a $(T,S)$-linear progress point on a geodesic $\alpha$ joining $g$ to $h$.
Then, Proposition~\ref{epsilonAnconaAcylindricallyhyperbolic} shows there exists $R\geq 0$ such that
$$G(g,h;B_R(p)^c)\leq \frac{1}{2} G(g,h).$$
Decomposing a trajectory for the random walk from $g$ to $h$ according to its first visit to $B_R(p)$, we get
$$G(g,h)=G(g,h;B_R(p)^c)+\sum_{q\in B_R(p)}G(g,q;B_R(p))G(q,h)$$
so that
$$G(g,h)\leq 2 \sum_{q\in B_R(p)}G(g,q)G(q,h).$$
Since $q$ is within $R$ of $p$, there exists a constant $C$ depending only on $R$ such that
$G(g,q)\leq C G(g,p)$ and $G(q,h)\leq C G(p,h)$.
We thus get
$$G(g,h)\leq 2 C \mathrm{Card}(B_R(e))G(g,p)G(p,h).$$
This proves one of the two inequalities.
According to~(\ref{triangleGreen'}), the other inequality is always satisfied, whether $p$ is linear progress point on $\alpha$ or not.
This concludes the proof.
\end{proof}

\subsection{Morse-Ancona inequalities in HHGs}

We will use in this section the results of Abbott, Behrstock and Durham \cite{AbbottBehrstockDurham}.
For any non-elementary hierarchically hyperbolic group $\Gamma$, they construct a special hierarchical structure, which has nice properties.
In particular, the action of $\Gamma$ on the underlying space $C\mathbf{S}$ is a largest acylindrical action, where $\mathbf{S}$ is the maximal domain in $\mathfrak{S}$.
When we consider a hierarchically hyperbolic group, we will always implicitly consider this hierarchical structure and we will always implicitly consider this acylindrical action on $C\mathbf{S}$.

\begin{proposition}\label{propMorselinearprogresshierarchicallyhyperbolic}
Let $\Gamma$ be a non-elementary hierarchically hyperbolic group.
For every Morse gauge $N$, there exists $S,T$ such that the following holds.
Let $\alpha$ be an $N$-Morse geodesic.
Then any point on $\alpha$ is a $(T,S)$-linear progress point.
\end{proposition}

To prove this proposition, we will use the following result.
Recall the following definition from \cite{AbbottBehrstockDurham}.
\begin{definition}
Let $(X,\mathfrak{S})$ be a hierarchically hyperbolic space with maximal domain $\mathbf{S}\in \mathfrak{S}$.
Let $Y\subset X$ and $D>0$.
We say that $Y$ has $D$-bounded projection if for every $U\in \mathfrak{S}\setminus \{\mathbf{S}\}$, we have
$\mathrm{diam}_U(Y)<D$.
\end{definition}

\begin{proposition}\label{ABDTheoremE}\cite[Theorem~E]{AbbottBehrstockDurham}
Let $\Gamma$ be a non-elementary hierarchically hyperbolic group.
A geodesic $\alpha$ in $\Gamma$ is $N$-Morse if and only if it has $D$-bounded projections, where $N$ and $D$ determine each other.
\end{proposition}

Actually, Abbott, Behrstock and Durham prove in \cite[Theorem~E]{AbbottBehrstockDurham} that for every $D$ there exists an $N$ such that the conclusion of this proposition holds, so they only prove there that $D$ determines $N$.
However, the fact that $N$ determines $D$ is contained in the discussion in \cite[Section~6]{AbbottBehrstockDurham}.
Roughly speaking, it is a consequence of the fact that contracting implies stability.
Also, projections of infinite Morse geodesics on product regions is well defined, according to \cite[Lemma~6.5]{AbbottBehrstockDurham} and the same result holds for infinite Morse geodesics.

We can now prove Proposition~\ref{propMorselinearprogresshierarchicallyhyperbolic}.

\begin{proof}
Let $N$ be a Morse gauge and let $\alpha$ be an $N$-Morse geodesic.
Then, $\alpha$ has $D$-bounded projections for some $D$ that only depends on $N$.
Let $p$ be any point on $\alpha$ and let $q_1,q_2$ be two points on $\alpha$ such that $q_1,p,q_2$ are aligned in this order.
Let $s\geq D$.
Recall that for two real numbers $t,s$ $\{\{t\}\}_s=0$ if $t\leq s$ and $\{\{t\}\}_s=t$ otherwise.
The distance formula (see \cite{BehrstockHagenSisto2}) shows that there exists $K,C$ only depending on $D$ (thus only depending on $N$) such that
$$d(q_1,q_2)\leq K \sum_{U\in \mathfrak{S}}\left \{\left \{d_U(q_1,q_2)\right \}\right \}_s +C.$$
Since $s\geq D$ and $\alpha$ has $D$-bounded projections,
$$d(q_1,q_2)\leq K d_{\mathbf{S}}(q_1,q_2) +C.$$
Choose $S=C$ and $T=2K$.
Assuming that $d(q_i,p)\geq S$, we have
$C\leq \frac{1}{2}d(q_1,q_2)$, so that
$$d(q_1,q_2)\leq Td_{\mathbf{S}}(q_1,q_2),$$
which concludes the proof.
\end{proof}

We now prove the same result for relatively hyperbolic groups.
We first prove the following.
Recall that a non-elementary relatively hyperbolic acylindrically acts on the graph obtained coning-off the parabolic subgroups, see \cite[Proposition~5.2]{Osinacylindrical}.

\begin{lemma}\label{relativelyhyperbolicMorseboundedprojections}
Let $\Gamma$ be a relatively hyperbolic group whose parabolic subgroups have empty Morse boundary.
A geodesic $\alpha$ in $\Gamma$ is $N$-Morse if and only if it has $D$-bounded projections on parabolic subgroups, where $D$ and $N$ determine each other.
\end{lemma}

This lemma is a consequence of results in \cite[Section~5]{Tran}, although it is not stated like that there, so we give the proof for convenience.

\begin{proof}
Since the parabolic subgroups have empty Morse boundary, for fixed $N$, there cannot be arbitrarily large $N$-Morse geodesics starting at the same point.
Otherwise, one could extract a sub-sequence of these $N$-Morse geodesics, using the Arzel\`a-Ascoli theorem to construct a geodesic ray, like in the proof of \cite[Corollary~1.4]{Cordes} and according to \cite[Lemma~2.10]{Cordes}, this resulting geodesic ray would be $N$-Morse.
Since parabolic subgroups quasi-isometrically embed in $\Gamma$, \cite[Lemma~2.9]{Cordes} shows that a $N$-Morse geodesic in the group $\Gamma$ stays within a bounded distance of a $N'$-Morse geodesic in the parabolic group.
Thus, there cannot be $N$-Morse geodesics in $\Gamma$ starting at a fixed base-point and travelling arbitrarily long in parabolic subgroups.

Conversely, let $\alpha$ be a geodesic with $D$-bounded projections on parabolic subgroups.
Let $p_1,p_2$ be two points on $\alpha$ and let $\beta$ be a $(\lambda,c)$-quasi-geodesic from $p_1$ to $p_2$.
We want to prove that any point of $\beta$ is within $N$ of a point of $\alpha$, where $N$ only depends of $\lambda$ and $c$.
Recall that the $M$-saturation of $\alpha$ is the union of $\alpha$ and all left cosets of parabolic subgroups whose $M$-neighborhood intersects $\alpha$.
First, according to \cite[Theorem~1.12~(1)]{DrutuSapir}, any point $x$ on $\beta$ is within $M_1$ of a point in the $M_0$-saturation of $\alpha$, where $M_0$ and $M_1$ only depends on $\lambda$ and $c$ (actually, $M_0$ does not even depend on those parameters, but ony on the group).
We just need to deal with the case where $x$ is within $M_1$ of a point $y$ in some left coset $gP$, where $P$ is a parabolic subgroup and such that $\alpha$ enters the $M_0$-neighborhood of $gP$.
Let $M_2=\max (M_0,M_1)$ so that both $\alpha$ and $\beta$ enter the $M_2$-neighborhood of $gP$ that we denote by $\mathcal{N}_{M_2}(gP)$.
Let $\alpha_1$ and $\alpha_2$, respectively $\beta_1$ and $\beta_2$ be the first and last points in $\mathcal{N}_{M_2}(gP)$ for $\alpha$, respectively $\beta$.
Then, according to \cite[Lemma~1.13~(a)]{Sistoprojections}, $\alpha_1$ and $\beta_1$ are within a bounded distance, say $\Lambda$, of the projection $q_1$ of $p_1$ on $gP$.
Similarly, $\alpha_2$ and $\beta_2$ are within $\Lambda$ of the projection $q_2$ of $p_2$ on $gP$.
Again, $\Lambda$ only depends on $\lambda$ and $c$.
Since $\alpha$ has $D$-bounded projections, $d(q_1,q_2)\leq D$ and so the distance between $\beta_1$ and $\beta_2$ is bounded.
In particular, since $\beta$ is a $(\lambda,c)$-quasi-geodesic, the distance between $x$ and $\beta_1$ is bounded.
Finally, $d(\beta_1,\alpha_1)\leq 2\Lambda$, so the distance between $x$ and $\alpha_1$ is bounded and the bound only depends on $\lambda$ and $c$.
This concludes the proof.
\end{proof}

We deduce the following from Lemma~\ref{relativelyhyperbolicMorseboundedprojections} and from the distance formula given by \cite[Theorem~0.1]{Sistoprojections}, exactly like we deduced Proposition~\ref{propMorselinearprogresshierarchicallyhyperbolic} from Proposition~\ref{ABDTheoremE} and the distance formula in hierarchically hyperbolic groups.

\begin{proposition}\label{propMorselinearprogressrelativelyhyperbolic}
Let $\Gamma$ be a non-elementary relatively hyperbolic groups, whose parabolic subgroups have empty Morse boundaries.
For every Morse gauge $N$, there exists $S,T$ such that the following holds.
Let $\alpha$ be an $N$-Morse geodesic.
Then any point on $\alpha$ is a $(T,S)$-linear progress point.
\end{proposition}

Let us briefly explain why we needed to work with hierarchically hyperbolic and relatively hyperbolic groups and not any acylindrically hyperbolic groups to get these results.
This will also explain why we needed to  use the modified hierarchical structure from \cite{AbbottBehrstockDurham}.
Consider the free group $\Gamma$ with two generators $a$ and $b$.
Then, $\Gamma$ is hyperbolic so that every geodesic is $N$-Morse for some fixed Morse gauge $N$.
Also, $\Gamma$ is hyperbolic relative to the subgroup generated by $a$ so it acylindrically acts on the graph obtained by coning-off this particular subgroup.
Choose now a geodesic $\alpha$ in $\Gamma$ travelling arbitrarily long in $\langle a \rangle$ so that $e$ is not a $(T,S)$-linear progress point on $\alpha$, whereas this geodesic is $N$-Morse.
One can make $e$ a $(T,S)$-linear progress point by considering the acylindrical action on the Cayley graph of $\Gamma$, so the apparent contradiction with our result comes from the fact that the first acylindrical action was not a largest one.

We deduce the following from Propositions~\ref{propMorselinearprogresshierarchicallyhyperbolic} and~\ref{propMorselinearprogressrelativelyhyperbolic} and from Proposition~\ref{epsilonAnconaAcylindricallyhyperbolic}.

\begin{proposition}\label{epsilonAnconaMorsegeodesic}
Let $\Gamma$ be a non-elementary hierarchically hyperbolic group or a non-elementary relatively hyperbolic group whose parabolic subgroups have empty Morse boundary.
Let $\mu$ be a probability measure on $\Gamma$ whose finite support generates $\Gamma$ as a semi-group.
Let $N$ be a Morse gauge.
Then, for any $\epsilon>0$, there exists $R>0$ such that the following hold.
Let $\alpha$ be an $N$-Morse geodesic and let $x,y,z$ be three points in this order on $\alpha$.
Then,
$$G(x,z;B_R(y)^c)\leq \epsilon G(x,z).$$
\end{proposition}

Finally, Theorem~\ref{maintheoremMorseAncona} is a consequence of Proposition~\ref{epsilonAnconaMorsegeodesic}, in the same way that Corollary~\ref{corollaryAnconalinearprogress} was a consequence of Proposition~\ref{epsilonAnconaAcylindricallyhyperbolic}.


\section{A map from the Morse boundary to the Martin boundary}\label{Sectionconstructionmap}
We consider a finitely generated group $\Gamma$ ans we assume that $\Gamma$ either is non-elementary hierarchically hyperbolic or non-elementary relatively hyperbolic whose parabolic subgroups have empty Morse boundaries.
We also consider a probability measure $\mu$ whose finite support generates $\Gamma$ as a semi-group.
We now construct a map from the Morse boundary to the Martin boundary.
We follow the strategy of \cite{KaimanovichErgodicity} and use deviation inequalities to prove that whenever $g_n$ is a sequence on a Morse geodesic going to infinity, then $g_n$ converges to some minimal point $\xi$ in the Martin boundary.

\subsection{Construction of the map}

In the following, we fix $x$ in the Morse boundary $\partial_{\mathcal{M}}\Gamma$ and we fix two Morse geodesic rays $\alpha$ and $\alpha'$, starting at $e$ and such that $[\alpha]=[\alpha']=x$.

\begin{lemma}\label{lemma1MorseinsideMartin}
Let $g_n$ be a sequence on $\alpha$ that converges to $x$ and that converges to some point $\xi$ in the Martin boundary $\partial_{\mu}\Gamma$.
Then $\xi$ is minimal.
\end{lemma}

\begin{proof}
Up to taking a sub-sequence, we can assume for simplicity that $g_n$ is an increasing sequence, meaning that $d(e,g_n)> d(e,g_m)$ if $n> m$.
We denote by $K_{\xi}$ the limit of $K(\cdot,g_n)$.
Let $H_{\xi}$ be the set of positive harmonic functions $h$ such that $\sup_{g\in \Gamma}\frac{h(g)}{K_{\xi}(g)}=1$.
The only thing to prove is that $H_{\xi}=\{K_{\xi}\}$.

Recall that whenever $f$ and $g$ are two functions such that there exists $C$ such that $\frac{1}{C}f\leq g \leq C f$, we write $f\asymp g$.
Since the $g_n$ all lie on a Morse geodesic,
Theorem~\ref{maintheoremMorseAncona}~shows that if $m<n$, then $G(e,g_n)\asymp G(e,g_m)G(g_m,g_n)$, so that
$K(g_m,g_n)\asymp \frac{1}{G(e,g_m)}$ and thus, for all $g\in \Gamma$,
$G(g,g_m)K(g_m,g_n)\asymp K(g,g_m)$.
Fixing $m$ and letting $n$ tend to infinity, we thus have
\begin{equation}\label{equationminimalpoint1}
    G(g,g_m)K_{\xi}(g_m)\asymp K(g,g_m).
\end{equation}

Letting $g,g'\in \Gamma$, recall that $F(g,g')$ denotes the probability of ever reaching $g'$, starting the random walk at $g$.
Also recall that $F(g,g')G(e,e)=G(g,g')$ (see also \cite[Lemma~1.13~(b)]{Woess}).
Then,~(\ref{triangleGreen}) shows that $G(g,g'')\geq F(g,g')G(g',g'')$ for any $g''$.
Letting $g''$ tend to infinity, we see that for any $\zeta$ in the Martin boundary,
$K_{\zeta}(g)\geq F(g,g')K_{\zeta}(g')$.
Thus, the Martin representation Theorem shows that
any positive harmonic function $h$ satisfies
$$h(g)\geq F(g,g')h(g').$$
Combining this inequality with~(\ref{equationminimalpoint1}), we obtain that there exists $C\geq 1$ such that if $h$ is a positive harmonic function, then
\begin{equation}\label{equationminimalpoint2}
    \forall m, \forall g\in \Gamma, h(g)\geq \frac{1}{C} K(g,g_m)\frac{h(g_m)}{K_{\xi}(g_m)}.
\end{equation}

We now argue by contradiction to prove that $H_{\xi}=\{K_{\xi}\}$ and consider $h\in H_{\xi}$ such that $h\neq K_{\xi}$.
Then, $h\leq K_{\xi}$ so that $h'=K_{\xi}-h$ also is harmonic and non-negative.
Moreover, by harmonicity, if it vanishes at some point, it vanishes everywhere, so $h'$ is in fact positive (see \cite[(1.15)]{Woess}).
We can thus apply~(\ref{equationminimalpoint2}) to $h'$ to get
\begin{equation}\label{equationminimalpoint3}
   \forall g\in \Gamma, h'(g)\geq \frac{1}{C}K_{\xi}(g)\underset{m\to \infty}{\limsup}\frac{h'(g_m)}{K_{\xi}(g_m)}.
\end{equation}
Since $h\in H_{\xi}$, by definition $\mathrm{inf}_{g\in \Gamma}\frac{h'(g)}{K_{\xi}(g)}=0$, so that~(\ref{equationminimalpoint3}) yields
$$\underset{m\to \infty}{\lim}\frac{h'(g_m)}{K_{\xi}(g_m)}=0,$$
hence
$$\underset{m\to \infty}{\lim}\frac{h(g_m)}{K_{\xi}(g_m)}=1.$$
We again use~(\ref{equationminimalpoint2}), but this time applied to $h$ itself and we let $m$ tend to infinity to obtain
\begin{equation}\label{equationminimalpoint4}
    \forall g\in \Gamma, h(g)\geq \frac{1}{C}K_{\xi}(g).
\end{equation}
Note that we necessarily have $C\geq 1$.
If $C=1$, then $h(g)\geq K_{\xi}(g)$ and since $h\in H_{\xi}$, we in fact have $h=K_{\xi}$, which is a contradiction.
Otherwise $\frac{1}{C}<1$ and we define $C_n=\frac{1}{C}\sum_{k=0}^n(1-\frac{1}{C})^k=1-(1-\frac{1}{C})^{n+1}$.
We prove by induction that for all $n$, $h\geq C_nK_{\xi}$.
The case $n=0$ is~(\ref{equationminimalpoint4}).
If the inequality is satisfied at $n$,
the function $h_n=\frac{1}{1-C_n}(h-C_nK_{\xi})$ also is in $H_{\xi}$ and we can apply~(\ref{equationminimalpoint4}) to get that
$$h_n\geq \frac{1}{C}K_{\xi},$$
so that
$$h\geq C_nK_{\xi}+\frac{1}{C}(1-C_n)K_{\xi}=C_{n+1}K_{\xi}.$$
Letting $n$ tend to infinity, we thus have $h\geq K_{\xi}$, which is again a contradiction.
\end{proof}

In other words, the closure of $\alpha$ in the Martin boundary is contained in the minimal Martin boundary.
We need this a priori minimality to prove the following lemma.

\begin{lemma}\label{lemma2MorseinsideMartin}
Let $g_n$ be a sequence on $\alpha$, converging to $x$ in the Morse boundary and converging to $\xi$ in the Martin boundary.
Let $g_n'$ be a sequence on $\alpha'$, also converging to $x$ in the Morse boundary and converging to $\xi'$ in the Martin boundary.
Then, $\xi=\xi'$.
\end{lemma}

\begin{proof}
We fix $m$.
Since $\alpha$ and $\alpha'$ are asymptotic geodesics, there exists a point $\tilde{g}_m$ on $\alpha'$ within a uniformly bounded distance from $g_m$.
In particular, for any $g\in \Gamma$ $G(g,\tilde{g}_m)\asymp G(g,g_m)$ and $G(\tilde{g}_m,g)\asymp G(g_m,g)$.
Since the Morse gauges of $\alpha$ and $\alpha'$ are fixed, we can use Theorem~\ref{maintheoremMorseAncona} and show that
if $n$ is large enough, then
$$G(e,g_n')\asymp G(e,\tilde{g}_m)G(\tilde{g}_m,g_n')\asymp G(e,g_m)G(g_m,g_n')$$
and
$$G(e,g_n)\asymp G(e,g_m)G(g_m,g_n).$$
In particular, $K(g_m,g_n')\asymp \frac{1}{G(e,g_m)}\asymp K(g_m,g_n)$.
Letting $n$ tend to infinity, we thus have
\begin{equation}\label{equationsamelimitMorsetoMartin}
    \forall m,K_{\xi}(g_m)\asymp K_{\xi'}(g_m).
\end{equation}

We now argue by contradiction to prove that $\xi=\xi'$.
First, Lemma~\ref{lemma1MorseinsideMartin} shows that both $K_{\xi}$ and $K_{\xi'}$ are minimal harmonic functions.
Thus, if we assume that $\xi\neq \xi'$, \cite[Lemma~1.7]{Anconapotentiel} shows that $\frac{K_{\xi'}(g_m)}{K_{\xi}(g_m)}$ converges to 0, when $m$ tends to infinity.
We thus get a contradiction with~(\ref{equationsamelimitMorsetoMartin}), so that $\xi=\xi'$.
\end{proof}

We can define a map $\Phi$ from the Morse boundary to the Martin boundary, sending a point $x$ to the unique point $\xi$ in the closure of $\alpha$ in the Martin boundary, where $\alpha$ is any Morse geodesic such that $[\alpha]=x$.
Uniqueness is given by Lemma~\ref{lemma2MorseinsideMartin}, which also shows that $\xi$ does not depend on the choice of $\alpha$.
Lemma~\ref{lemma1MorseinsideMartin} shows that $\xi$ is minimal.
We also denote by $\Phi_N$ the restriction of $\Phi$ to the $N$-Morse boundary.

\subsection{Injectivity}
We prove the following here.

\begin{proposition}
The map $\Phi$ is one-to-one.
\end{proposition}

\begin{proof}
Let $[\alpha]$ and $[\alpha']$ be two points in the Morse boundary.
Up to taking the maximum of the two Morse gauges, we can assume that $\alpha$ and $\alpha'$ are two $N$-Morse geodesic rays, starting at $e$, for some fixed $N$.
Let $\xi=\Phi(\alpha)$ and $\xi'=\Phi(\alpha')$ and let $g_n$ be a sequence on $\alpha$ that converges to $\xi\in \partial_{\mu}\Gamma$ and $g'_n$ a sequence on $\alpha'$ that converges to $\xi'\in \partial_{\mu}\Gamma$.
We will prove that $K_{\xi}(g_m)$ tends to infinity, whereas $K_{\xi'}(g_m)$ converges to 0, when $m$ tends to infinity.

Let $m\leq n$.
Theorem~\ref{maintheoremMorseAncona} shows that $G(e,g_n)\leq CG(e,g_m)G(g_m,g_n)$ for some fixed $C\geq 0$.
Thus, $K(g_m,g_n)\geq \frac{1}{C}\frac{1}{G(e,g_m)}$.
Letting $n$ tend to infinity, we see that
$K_{\xi}(g_m)\geq \frac{1}{C}\frac{1}{G(e,g_m)}$.
Since $g_m$ tends to infinity, $G(e,g_m)$ converges to 0.
This proves that $K_{\xi}(g_m)$ tends to infinity.

Now, let $\beta_{m,n}$ be a geodesic from $\gamma_m$ to $\gamma'_n$.
According to \cite[Lemma~2.3]{Cordes}, $\beta_{m,n}$ is $N'$-Morse for some $N'$ that only depends on $N$.
Since Morse triangles are thin (see \cite[Lemma~2.3]{CharneyCordesMurray}), $\beta_{m,n}$ passes within a uniformly bounded distance of $e$.
Using again Theorem~\ref{maintheoremMorseAncona}, there exists $C'$ such that $G(g_m,g'_n)\leq C'G(g_m,e)G(e,g'_n)$.
Thus, $K(g_m,g'_n)\leq C'G(g_m,e)$ and so $K_{\xi'}(g_m)\leq C'G(g_m,e)$.
Again, $G(g_m,e)$ converges to 0, hence so does $K_{\xi'}(g_m)$.

Consequently, we can find $m$ such that $K_{\xi}(g_m)\neq K_{\xi'}(g_m)$, so that $\xi\neq \xi'$.
\end{proof}

\subsection{Continuity}

Recall that we can endow the $N$-Morse boundary with a topology so that convergence is defined as follows.
A sequence $x_n\in \partial_{M,e}^N\Gamma$ converges to $x\in \partial_{M,e}^N\Gamma$ if there exists $N$-Morse geodesic rays $\alpha_n$ with $\alpha_n(0)=e$ and $[\alpha_n]=x_n$, such that
every sub-sequence of $\alpha_n$ contains a sub-sequence that converges uniformly on compact sets to a geodesic ray $\alpha$ with $[\alpha]=x$, see \cite{Cordes} for more details.

\begin{proposition} \label{prop:topological embedding}
Let $N$ be a Morse gauge.
The map $\Phi_N:\partial^N_{M,e}\Gamma\rightarrow \partial_{\mu}\Gamma$ is a topological embedding.
\end{proposition}

\begin{proof}
Let $x_n$ converge to $x$ in $\partial_{M,e}^N\Gamma$ and $\xi_n=\Phi_N(x_n)$, $\xi=\Phi_N(x)$.
Let $\alpha_n$ be Morse geodesics starting at $e$ as above, that is, from every subsequence of $\alpha_n$, one can extract a sub-sequence that converges to some $N$-Morse geodesic $\alpha$, uniformly on compact sets and where $[\alpha_n]=x_n$, $[\alpha]=x$.
We want to prove that $\xi_n$ converges to $\xi$.
Since the Martin boundary is compact, we only have to prove that $\xi$ is the only limit point of $\xi_n$.
We first take a sub-sequence $\alpha_{\sigma_1(n)}$ such that that $\xi_{\sigma_1(n)}$ converges to some $\xi'$ and we now prove that $\xi'=\xi$.
We already know that $\xi$ is minimal, so we only have to prove that $K_{\xi'}\leq K_{\xi}$.

We take a sub-sub-sequence $\alpha_{\sigma_2(n)}$ and we take $\alpha$ respectively representing $x_{\sigma_2(n)}$ and $x$, such that $\alpha_{\sigma_2(n)}$ converges to $\alpha$, uniformly on compact sets.
For simplicity, we write $\alpha_n=\alpha_{\sigma_2(n)}$.
Then, there is a sequence of points $g_n$ on $\alpha$ going to infinity such that the geodesic $\alpha_n$ fellow travels with $\alpha$ up to $g_n$.

Let $g \in \Gamma$.
We want to prove that $K_{\xi'}(g)\leq K_{\xi}(g)$.
According to \cite[Lemma~2.8]{Cordes}, one can find an $N'$-Morse geodesic $\beta$, starting at $g$, which is asymptotic to $\alpha$, with $N'$ that only depends on $N$ and $g$.
Similarly, one can find $N'$-Morse geodesics $\beta_n$, starting at $g$ and asymptotic to $\alpha_n$.
Recall that Morse triangles are thin, see \cite[Lemma~2.3]{CharneyCordesMurray}.
Thus, $\beta$ eventually lies a bounded distance away from $\alpha$.
More precisely, the geodesic $\beta$ roughly travels from $g$ to its projection $\hat{g}$ on $\alpha$ and then fellow travel with $\alpha$.
Similarly, for large enough $n$, the geodesics $\beta_n$ roughly travel from $g$ to $\hat{g}$, then fellow travel with $\alpha$ up to $g_n$ and then fellow travel with $\alpha_n$.

To sum up, for every large enough $n$, we can find $g_n$ on $\alpha$ within a uniformly bounded distance of a point on $\beta$, a point on $\beta_n$ and a point on $\alpha_n$.
Let us call $g_n'$ the corresponding point on $\alpha_n$.
Also, notice that $g_n$ goes to infinity, see the picture at the end of the proof.

We fix $\epsilon>0$.
Since $N'$ is fixed (although it depends on $g$), Proposition~\ref{epsilonAnconaMorsegeodesic} shows that
there exists $R$ such that for every point $\tilde{g}_n$ on $\alpha_n$ such that $e,g_n',\tilde{g}_n$ lie in this order on $\alpha_n$,
\begin{equation}\label{continuity1}
    G(g,\tilde{g}_n,B_R(g_n)^c)\leq \epsilon G(g,\tilde{g}_n).
\end{equation}
Decomposing a path from $g$ to $\tilde{g}_n$ according to its last visit to $B_R(g_n)$, we have
\begin{equation}\label{continuity2}
\begin{split}
    G(g,\tilde{g}_n)&=G(g,\tilde{g}_n;B_R(g_n)^c)\\
    &+\sum_{u\in B_R(g_n)}G(g,u)G(u,\tilde{g}_n;B_R(g_n)^c).
\end{split}
\end{equation}

Since $g_n$ goes to infinity, it converges to $\xi$ in the Martin boundary.
Let us prove that for large enough $n$, for any $u\in B_R(g_n)$,
$K(g,u)\leq K_{\xi}(g)+\epsilon$.
Assume by contradiction this is not the case, so that in particular, there is a sequence $u_n\in B_R(g_n)$ such that
$K(g,u_n)$ does not converge to $K_{\xi}(g)$.
Up to taking a sub-sequence, we can assume that $u_n$ converges to some point $\xi''$ in the Martin boundary.
Since, $d(u_n,g_n)$ is uniformly bounded, for any $g'$,
$K(g',u_n)\leq CK(g',g_n)$, for some constant $C$ and similarly, $K(g',g_n)\leq CK(g',u_n)$.
This proves that $K_{\xi''}(g')\asymp K_{\xi}(g')$ and since $\xi$ is minimal, we have $\xi''=\xi$.
In particular, $K(g,u_n)$ converges to $K_{\xi}(g)$, so we get a contradiction.

Thus, for large enough $n$, we have, for any $u\in B_R(g_n)$,
$$G(g,u)\leq G(e,u)(K_{\xi}(g)+\epsilon).$$
Using~(\ref{continuity1}) and~(\ref{continuity2}), we get
$$(1-\epsilon)G(g,\tilde{g}_n)\leq (K_{\xi}(g)+\epsilon)\sum_{u\in B_R(g_n)}G(e,u)G(u,\tilde{g}_n;B_R(g_n)^c).$$
Decomposing now a path from $e$ to $\tilde{g}_n$ according to its last visit to $B_R(g_n)$, we see that
$$\sum_{u\in B_R(g_n)}G(e,u)G(u,\tilde{g}_n;B_R(g_n)^c)\leq G(e,\tilde{g}_n),$$
hence $(1-\epsilon)G(g,\tilde{g}_n)\leq (K_{\xi}(g)+\epsilon)G(e,\tilde{g}_n)$ and so
\begin{equation}\label{continuity3}
(1-\epsilon)K(g,\tilde{g}_n)\leq K_{\xi}(g)+\epsilon.
\end{equation}

Let us now prove that we can find such a sequence $\tilde{g}_n$ on $\alpha_n$ such that $\tilde{g}_n$ converges to $\xi'$ in the Martin boundary.
Indeed, for fixed $n$, any sequence $\tilde{g}_{n,m}$ on $\alpha_n$ that goes to infinity when $m$ goes to infinity converges to $\xi_n$.
Now, recall that the Martin compactification is metrizable.
Let us fix an arbitrary distance $d_\mu$ on it so that for every $n$, for every large enough $m$ (depending on $n$),
$d_{\mu}(\tilde{g}_{n,m},\xi_n)\leq \frac{1}{n}$.
Taking $m$ large enough, we can assume that $e,g_n',\tilde{g}_{n,m}$ do lie in this order on $\alpha_n$.
For every $n$, let us fix such an $m_n$ and write $\tilde{g}_n=\tilde{g}_{n,m_n}$.
Then, $d_{\mu}(\tilde{g}_n,\xi_n)\leq \frac{1}{n}$ and since $\xi_n$ converges to $\xi'$, so does $\tilde{g}_n$.

Letting $n$ tend to infinity in~(\ref{continuity3}), we get
$(1-\epsilon)K_{\xi'}(g)\leq K_{\xi}(g)+\epsilon$ and since $\epsilon$ is arbitrary, we have
$K_{\xi'}(\gamma)\leq K_{\xi}(\gamma)$.
Finally, since $\xi$ is minimal, this proves that $\xi'=\xi$, which proves the $\Phi_N$ is continuous.

Since $\partial^N_{M,e}\Gamma$ is compact and $\partial_\mu \Gamma$ is Hausdorff follows from the closed map lemma that $\Phi_N$ is a topological embedding.
\end{proof}

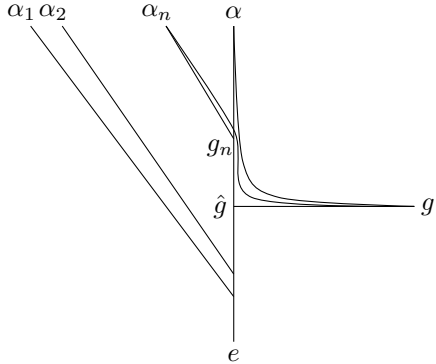
\begin{center}
\begin{tikzpicture}[scale=.6]
\draw (0,0)--(0,7);
\draw (0,-.3) node{$e$};
\draw (0,7.3) node{$\alpha$};
\draw (0,1)--(-4.5,7);
\draw (-4.7,7.3) node{$\alpha_1$};
\draw (0,1.5)--(-3.8,7);
\draw (-4,7.3) node{$\alpha_2$};
\draw (0,4.5)--(-1.5,7);
\draw (-1.7,7.3) node{$\alpha_n$};
\draw (4,3)--(0,3);
\draw (4.3,3) node{$g$};
\draw (-.3,4.3) node{$g_n$};
\draw (-.3,3) node{$\hat{g}$};
\draw plot [smooth] coordinates {(4,3) (1,3.2) (.2,4) (0,7)};
\draw plot [smooth] coordinates {(4,3) (.4,3.2) (0,4.7) (-1.5,7)};
\end{tikzpicture}
\end{center}

All this concludes the proof of Theorem~\ref{maintheoremMorsetoMartin}, for the map $\Phi$ is continuous with respect to the direct limit topology if and only if all the maps $\Phi_N$ are continuous for fixed $N$. \qed

A small application of Proposition \ref{prop:topological embedding} sheds some light on the topology of the Martin boundary of the mapping class group, $\partial_\mu \mathrm{MCG}(S)$.
\begin{corollary}\label{corollaryMCGs}
For any $n \geq 2$ there exists a surface of finite type $S$ such that $\partial_\mu \mathrm{MCG}(S)$ contains a topologically embedded $(n-1)$-sphere.
\end{corollary}
\begin{proof}
By \cite[Corollary 4.4]{Cordes} for any $n \geq 2$ there exists a surface of finite type $S$ such that the Morse boundary of the Teichm\"uller space $Teich(S)$ contains a topologically embedded $(n-1)$-sphere, and by \cite[Theorem 4.12]{Cordes} the Morse boundary of $Teich(S)$ coincides with the Morse boundary of $\mathrm{MCG}(S)$.
The result now follows from Theorem~\ref{maintheoremMorsetoMartin}.
\end{proof}


\section{Measure of the Morse boundary}\label{Sectionmeasurezero}

Consider a transient random walk on a finitely generated group $\Gamma$.
As explained in the introduction, the random walk almost surely converges to a point $X_{\infty}$ in the Martin boundary.
Denoting the law of $X_{\infty}$ by $\nu$, $(\partial_\mu\Gamma,\nu)$ is a model for the Poisson boundary.
The measure $\nu$ is called the harmonic measure.
See \cite{Sawyer} for more details.

Recall that a measure $\kappa$ on a measurable space $X$ on which $\Gamma$ acts is called $\mu$-stationary if $\mu * \kappa=\kappa$, where by definition, for every measurable set $A\subset X$,
$$\mu*\kappa(A)=\sum_{g\in \Gamma}\mu(g)\kappa(g^{-1}A).$$

The following is folklore.
Since the proof is very short, we write it for convenience.
\begin{lemma}\label{lemmaharmonicstationary}
The harmonic measure $\nu$ on the Martin boundary is $\mu$-stationary.
\end{lemma}

\begin{proof}
Since $\nu$ is the exit law of the random walk, we have
$$\nu(A)=P(X_\infty \in A)=\sum_{g\in \Gamma}P(X_\infty \in A |X_1=g)P(X_1=g).$$
By definition, $\mu$ is the law of $X_1$ so that
$$\nu(A)=\sum_{g\in \Gamma}P(X_1^{-1}X_\infty \in g^{-1}A|X_1=g)\mu(g).$$
Notice that $X_1^{-1}X_\infty$ is the limit of $X_1^{-1}X_n$, hence it has the same law as $X_\infty$.
Also, $X_1^{-1}X_\infty$ is independent of $X_1$.
We thus get
$$\nu(A)=\sum_{g\in \Gamma}P(X_1^{-1}X_\infty\in g^{-1}A)\mu(g)=\sum_{g\in \Gamma}\nu(g^{-1}A)\mu(g)=\mu*\nu(A).$$
This concludes the proof.
\end{proof}

\subsection{The Morse boundary has zero harmonic measure}
As we saw, if $\Gamma$ is non-elementary relatively hyperbolic with parabolic subgroups having empty Morse boundary or if $\Gamma$ is non-elementary hierarchically hyperbolic, there is a map from the Morse boundary to the Martin boundary.
Our goal is to prove Theorem~\ref{maintheoremmeasureMorse}, that is unless $\Gamma$ is hyperbolic, its Morse boundary has measure 0 with respect to the harmonic measure.

The following elementary result will be useful.
\begin{lemma}\label{countableunion}

There exists a countable collection $\mathcal{F}$ of Morse gauges such that for any $\Gamma$, $\partial_{M}\Gamma=\cup_{N\in \mathcal{F}}\partial^{N}_{M}\Gamma$.
Moreover $\partial_{M}\Gamma \times \partial_{M}\Gamma \setminus Diag=\cup_{N\in \mathcal{F}}\Omega_N$ where $\Omega_N$ is the set of couples $(x_1,x_2)\in \partial_M\Gamma\times \partial_M\Gamma$ such that there exists a bi-infinite $N$-Morse geodesic from $x_1$ to $x_2$.
\end{lemma}

\begin{proof}
The Morse boundary is defined as the union of all $N$-Morse boundaries. 
According to Proposition~\ref{ABDTheoremE} and Lemma~\ref{relativelyhyperbolicMorseboundedprojections} for every Morse gauge $N$ there is an integer $D(N)>0$ such that any geodesic with $D(N)$ bounded projections is $N$-Morse. 
Furthermore any Morse geodesic has bounded projections proving the first claim.
For the second claim, according to \cite[Proposition~3.11]{Cordes}, 
$\partial_M\Gamma\times \partial_M\Gamma\setminus Diag$ is the union of $\Omega_N$ over all Morse gauges, and we can conclude similarly.
\end{proof}
Now, we prove the following result, which gives sense to $\nu(\partial_M\Gamma)$.

\begin{lemma}\label{LemmaMorseborelian}
Let $\Gamma$ be a non-elementary hierarchically hyperbolic group or a non-elementary relatively hyperbolic group whose parabolic subgroups have empty Morse boundary.
Then, the Morse boundary
$\partial_{M}\Gamma\subset \partial_{\mu}\Gamma$ is a Borel subset of the Martin boundary.
\end{lemma}

\begin{proof}
By Lemma \ref{countableunion} the Morse boundary can be obtained as a countable union of spaces $\partial_M^N\Gamma$.
We just need to prove that for fixed $N$, the image of $\partial_M^N\Gamma$ is a borelian subset of the Martin boundary.
This follows immediately from the fact that the map $\Phi$ from the Morse boundary to the Martin boundary is continuous and the fact that for fixed $N$, $\partial_M^N\Gamma$ is compact, according to \cite[Proposition~3.12]{Cordes}.
\end{proof}

In fact, we have another realization of the Morse boundary as a subset of a topological model for the Poisson boundary.
Recall from the introduction that a hierarchically hyperbolic group acylindrically acts on a hyperbolic space $C\mathbf{S}$ and that a relatively hyperbolic group acylindrically acts on the coned-off graph with respect to the relatively hyperbolic structure.
For any such group $\Gamma$, we will denote by $X$ the corresponding hyperbolic space on which $\Gamma$ acylindrically acts.
We endow $\Gamma\cup \partial_M\Gamma$ with the direct limit topology and we endow $\Gamma\cup \partial X$ with the usual topology coming from the Gromov product on $X$.

\begin{lemma}\label{inclusionMorseGromovboundary}\cite[Corollary~6.1,~Lemma~6.5]{AbbottBehrstockDurham}\cite[Theorem~7.6]{CashenMacKay}
Let $\Gamma$ be either a non-elementary hierarchically hyperbolic group or a non-elementary relatively hyperbolic group whose parabolic subgroups have empty Morse boundaries.
Then, there is an embedding $\Psi: \partial_M\Gamma \to \partial X$.
Moreover, a sequence $g_n$ in $\Gamma$ converge to $x\in \partial_M\Gamma$ if and only if $g_n$ converges to $\Psi(x)\in \partial X$.
\end{lemma}

According to \cite[Theorem~1.1, Theorem~1.5]{MaherTiozzo}, the random walk almost surely converges to a point in $\partial X$ and $\partial X$ endowed with the exit measure is a realization of the Poisson boundary.
In the following, we will denote by $\nu$ the harmonic measure on the Martin boundary $\partial_\mu\Gamma$ and by $\nu_X$ the harmonic measure on $\partial X$ to avoid confusion.
We now prove that the Morse boundary has full measure with respect to one harmonic measure if and only if it has full measure with respect to the other.

\begin{proposition}\label{propequivalencePoissonboundaries}
Let $\Gamma$ be either a non-elementary hierarchically hyperbolic group or a non-elementary relatively hyperbolic group whose parabolic subgroups have empty Morse boundaries.
Then, $\nu(\Phi(\partial_M\Gamma))=1$ if and only if $\nu_X(\Psi(\partial_M\Gamma))=1$.
\end{proposition}

\begin{proof}
Assume first that $\nu_X(\Psi(\partial_M\Gamma))=1$.
Then, almost surely, the random walk $X_n$ converges to a point $\Psi(x_\infty)$.
According to Lemma~\ref{inclusionMorseGromovboundary}, $X_n$ almost surely converges to $x_\infty$ with respect to the topology on $\Gamma \cup \partial_M\Gamma$.
By construction of the map $\Phi$, this implies that $X_n$ almost surely converges to $\Phi(x_\infty)$ in the Martin boundary.
In particular, $\nu(\Phi(\partial_M\Gamma))=1$.

Proving the converse is more difficult, because we do not know if converging to $\Phi(x_\infty)$ in the Martin boundary implies converging to $x_\infty$ in the Morse boundary, so we have to find another strategy.
Since the map $\Phi$ is one-to-one, we can define the inverse map $\Phi^{-1}:\Phi(\partial_M\Gamma)\to \partial_M\Gamma$.
It is not clear if this inverse map is continuous, but it is measurable as we now prove.

\begin{lemma}\label{Phi-1measurable}
The map $\Phi^{-1}:\Phi(\partial_M\Gamma)\to \partial_M\Gamma$ is measurable.
\end{lemma}

\begin{proof}
By Lemma \ref{countableunion} the Morse boundary is a countable union of $N$-Morse boundaries $\partial_M^N\Gamma$.
Let $A$ be a borelian subset of $\partial_M\Gamma$.
Since $\partial_M^N\Gamma$ is closed in $\partial_M\Gamma$, $\partial_M^N\Gamma\cap A$ also is measurable.
Moreover, when restricted to $\partial_M^N\Gamma$ which is compact, the map $\Phi$ is an embedding, so that $\Phi(\partial_M^N\Gamma \cap A)$ is a borelian subset of the Martin boundary.
Thus,
$$(\Phi^{-1})^{-1}(A)=\bigcup_N\Phi(\partial_M^N\Gamma\cap A)$$
is measurable, which proves the lemma.
\end{proof}

Since $\nu(\Phi(\partial_M\Gamma))=1$, we can see $\nu$ as a measure on $\Phi(\partial_M\Gamma)$.
Consider the pushforward measure $\tilde{\nu}_X=\Psi \circ \Phi^{-1}_*\nu$.
By definition, this is a probability measure on $\partial X$ such that $\tilde{\nu}_X(\Psi (\partial_M\Gamma))=1$.
Recall that $\nu$ is stationary on $\partial_\mu\Gamma$.
Notice that both maps $\Phi$ and $\Psi$ are $\Gamma$-equivariant, hence $\tilde{\nu}_X$ is stationary on $\partial X$.
According to \cite[Theorem~1.1]{MaherTiozzo}, $\nu_X$ is the only stationary probability measure on $\partial X$, so
$\tilde{\nu}_X=\nu_X$.
This proves that $\nu_X(\Psi(\partial_M\Gamma))=1$, which concludes the proof.
\end{proof}

\subsection{Double ergodicity of the harmonic measure on the Martin boundary}
We prove here Theorem~\ref{maintheoremmeasureMorse} in the particular case where the measure $\mu$ is symmetric.
In general, letting $\mu$ be a probability measure on $\Gamma$, denote by $\check{\mu}$ the reflected measure, that is, $\check{\mu}(\gamma)=\mu(\gamma^{-1})$.
This reflected measure also gives rise to a random walk, hence to a Martin boundary $\partial_{\check{\mu}}\Gamma$ and a harmonic measure $\check{\nu}$.
We call $\check{\nu}$ the reflected harmonic measure.
We have the following.

\begin{theorem}\label{ergodicityPoissonboundary}\cite[Theorem~6.3]{KaimanovichPoissonhyperbolic}
The product measure $\nu\otimes \check{\nu}$ is ergodic for the diagonal action of $\Gamma$ on $\partial_{\mu}\Gamma\times \partial_{\check{\mu}}\Gamma$.
\end{theorem}

Let $A$ be a $\Gamma$-invariant borelian subset of the Martin boundary $\partial_\mu\Gamma$.
Then, $A\times \partial_\mu\Gamma$ is also invariant and so $\nu\otimes \check{\nu}(A\times \partial_\mu\Gamma)$ is 0 or 1.
In particular, the harmonic measure $\nu$ is ergodic for the group action on $\partial_{\mu}\Gamma$.
Similarly, $\check{\nu}$ is ergodic.

The group $\Gamma$ acts on its Morse boundary, sending a Morse geodesic ray to another Morse geodesic ray.
By construction, the map $\Phi:\partial_M\Gamma \rightarrow\partial_{\mu}\Gamma$ is equivariant, so that the image of the Morse boundary is $\Gamma$-invariant inside the Martin boundary.
By ergodicity, we get
$\nu(\partial_{M}\Gamma) = 0 \text{ or } 1$
and
$\nu\otimes \check{\nu}(\partial_{M}\Gamma) = 0 \text{ or } 1$.
Recall that we want to prove that $\nu(\partial_M\Gamma)=0$ unless the group is hyperbolic.

\begin{proposition}
Let $\Gamma$ be a non-elementary hierarchically hyperbolic group or a non-elementary relatively hyperbolic group whose parabolic subgroups have empty Morse boundary.
Let $\mu$ be a probability measure on $\Gamma$ whose finite support generates $\Gamma$ as a semi-group.
Let $\nu$ be the corresponding harmonic measure on the Martin boundary and $\check{\nu}$ be the reflected harmonic measure.
If $\Gamma$ is not hyperbolic, then $\nu\otimes \check{\nu}(\partial_M\Gamma\times \partial_M\Gamma)=0$.
\end{proposition}

\begin{proof}
We argue by contradiction, so we assume in the following that
\begin{enumerate}
\item $\Gamma$ is non-elementary relatively hyperbolic with parabolic subgroups having empty Morse boundary or non-elementary hierarchically hyperbolic,
\item $\Gamma$ is not hyperbolic,
\item $\nu\otimes \check{\nu}(\partial_M\Gamma\times \partial_M\Gamma)\neq 0$.
\end{enumerate}

For a fixed Morse gauge $N$, let $\Omega_N$ be the set of couples $(x_1,x_2)\in \partial_M\Gamma\times \partial_M\Gamma$ such that there exists a bi-infinite $N$-Morse geodesic from $x_1$ to $x_2$.
By Lemma \ref{countableunion}
$$\partial_M\Gamma\times \partial_M\Gamma \setminus \mathrm{Diag} =\bigcup_{N}\Omega_N.$$
Moreover, one can choose a countable number of Morse gauges in this union, as in the proof of Lemma~\ref{LemmaMorseborelian}.
Thus, to find a contradiction, it is sufficient to prove that for every $N$, $\nu\otimes \check{\nu}(\Omega_N)=0$.

Since we assume that $\nu\otimes \check{\nu}(\partial_M\Gamma\times \partial_M\Gamma)\neq 0$, in particular, $\nu(\partial_M\Gamma)\neq 0$ and $\check{\nu}(\partial_M\Gamma)\neq 0$.
Define $Y\subset \Gamma$ to be either an infinite parabolic subgroup if $\Gamma$ is assumed to be non-elementary relatively hyperbolic or one of the hyperbolic spaces $CU$, $U\in \mathfrak{S}\setminus S$ if $\Gamma$ is assumed to be non-elementary hierarchically hyperbolic.
For every $g\in Y$, let
$\partial_M^{g}\Gamma$ be the set of Morse geodesic rays $x\in \partial_M\Gamma$ which project within distance $D_0$ of $g$, where $D_0$ is a fixed constant.
Projection of a Morse geodesic ray on $Y$ is well defined up to a bounded distance, so if $D_0$ is chosen large enough, then $\partial_M^{g}\Gamma$ is well defined.
Then,
$$\partial_M\Gamma=\bigcup_{g\in Y}\partial_M^{g}\Gamma,$$
so that there exists $g_0$ such that $\nu(\partial_M^{g_0}\Gamma)\neq 0$.

Fix $g\in Y$.
For any Morse geodesic ray $\alpha$ representing some $x\in \partial_M^{g_0}\Gamma$, the translated geodesic $g\alpha$ represents $g\cdot x\in \partial_M\Gamma$.
Moreover, $g\cdot x$ projects on $Y$ within a uniformly bounded distance of $gg_0$, so up to enlarging $D_0$ (not depending on $g$), $g\cdot x\in \partial_M^{gg_0}\Gamma$.
Thus, $g\cdot \partial_M^{g_0}\Gamma\subset \partial_M^{gg_0}$.
Recall that the random walk almost surely converges to a point $X_{\infty}$ in the Martin boundary and that $\nu$ is the law of $X_{\infty}$, so that $P(X_\infty\in \partial_M^{g_0}\Gamma)>0$.
Translating everything by $g$, we get
$P_g(X_{\infty}\in \partial_M^{gg_0}\Gamma)>0$,
where $P_g$ denotes the probability measure on the set of sample paths for the random walk starting at $g$.

Since the random walk generates $\Gamma$ as a semi-group, there exists $n_g$ such that $P(X_{n_g}=g)>0$.
Concatenating a trajectory of the random walk from $e$ to $g$ and an infinite trajectory starting at $g$ yields an infinite trajectory starting at $e$.
We thus get $P(X_{\infty}\in \partial_M^{gg_0}\Gamma)>0$ and this holds for every $g$.
In particular, applying this for $gg_0^{-1}$, we get
$\nu(\partial_M^{g}\Gamma)>0$.
Similarly, for every $g$,
$\check{\nu}(\partial_M^g\Gamma)>0$, so that for every $g_1,g_2$,
$\nu\otimes \check{\nu}(\partial_M^{g_1}\Gamma\times \partial_M^{g_2}\Gamma)>0$.

Recall that we want to prove that for every $N$, $\nu\otimes \check{\nu}(\Omega_N)=0$.
Fix a Morse gauge $N$.
Then there exists $D$ such that every $N$-Morse geodesic has $D$-bounded projections.
Choose two points $g_1,g_2$ on $Y$ such that $d(g_1,g_2)$ is bigger than $D$.
For every $(x_1,x_2)\in \partial_M^{g_1}\Gamma\times \partial_M^{g_2}\Gamma$, any Morse geodesic from $x_1$ to $x_2$ cannot have $D$-bounded projections, so $(x_1,x_2)\notin \Omega_N$.
This proves that $\partial_M^{g_1}\Gamma\times \partial_M^{g_2}\Gamma\subset \Omega_N^c$.
In particular,
$\nu\otimes \check{\nu}(\Omega_N)<1$.
Finally, notice that $\Omega_N$ is $\Gamma$-invariant.
Since $\nu\otimes \check{\nu}$ is ergodic, we necessarily have
$\nu\otimes \check{\nu}(\Omega_N)=0$.
This concludes the proof.
\end{proof}

Assuming that $\mu$ is symmetric, $\mu=\check{\mu}$, so that $\nu=\check{\nu}$.
This proves Theorem~\ref{maintheoremmeasureMorse} for symmetric measures $\mu$.

\subsection{Another proof of Theorem~\ref{maintheoremmeasureMorse}}
We do not assume anymore that $\mu$ is symmetric.
We again argue by contradiction and consider a group $\Gamma$ such that
\begin{enumerate}
\item $\Gamma$ is non-elementary relatively hyperbolic with parabolic subgroups having empty Morse boundary or non-elementary hierarchically hyperbolic,
\item $\Gamma$ is not hyperbolic,
\item $\nu(\partial_M\Gamma)\neq 0$.
\end{enumerate}
By ergodicity, $\nu(\partial_M\Gamma)=1$ and so
according to Proposition~\ref{propequivalencePoissonboundaries}, we also have $\nu_X(\partial_M\Gamma)>0$ (notice that we did not need to use this proposition in our first proof).
Since $\partial_M\Gamma$ can be described as a countable union of $N$-Morse boundaries, there exists $\eta>0$ and there exists a Morse gauge $N$ such that
$\nu_X(\partial_M^N\Gamma)\geq \eta$.
In particular, the random walk $X_n$ converges to a $N$-Morse point $X_\infty$ in $\partial X$ with probability at least $\eta$.
Recall the following terminology from \cite{MathieuSisto}.

\begin{definition}
Let $\Gamma$ be a finitely generated group acylindrically acting on a hyperbolic space $X$.
Let $(g_0,g_1,...,g_n,...)$ be a sequence (finite or infinite) of points of $\Gamma$.
We say that the sequence is tight around $g_k$ at scale $l$ (with respect to a constant $C$) if for every $k_1\leq k\leq k_2$ with $k_2-k_1\geq l$, we have
\begin{enumerate}
\item $d_X(g_{k_2},g_{k_1})\geq (k_2-k_1)/C$,
\item the length of the path $(g_{k_1},...,g_{k_2})$ in $\Gamma$ is at most $C(k_2-k_1)$
\item $d_{\Gamma}(g_{k'},g_{k'+1})\leq \max (l,|k-k'|/C)$ for every $k'$.
\end{enumerate}
\end{definition}

Clearly, for an infinite sequence $(g_0,...,g_n,...)$ and for every $k\leq n$, if $(g_0,...,g_n,...)$ is tight around $g_k$ at scale $l$, then the same holds for the finite sequence $(g_0,...,g_n)$.

\begin{lemma}\label{tightaroundrandomwalk}\cite[Lemma~10.11]{MathieuSisto}
There exists $C$ such that for every $k$, for every $l$,
the probability that the infinite sequence $(e,X_1,...,X_n,...)$ is tight around $X_k$ at scale $l$ is at least $1-C\mathrm{e}^{-l/C}$.
\end{lemma}

Actually, \cite[Lemma~10.11]{MathieuSisto} is only about the finite sequence $(e,X_1,...,X_n)$ for $k\leq n$, but it is clear from the proof that the result holds for the infinite sequence $(e,X_1,...,X_n,...)$.

\medskip
We finish the proof of Theorem~\ref{maintheoremmeasureMorse}.
We choose $l$ so that $1-C\mathrm{e}^{-l/C}\leq \eta/2$.
Then, the probability that $X_n$ converges to a $N$-Morse point $X_\infty$ and that the infinite path $(e,X_1,...,X_n,...)$ is tight around $X_k$ at scale $l$ is at least $\eta/2$.
According to \cite[Lemma~10.12]{MathieuSisto}, if $(e,X_1,...,X_n)$ is tight around $k$ at scale $l$, then the distance in $\Gamma$ between $X_k$ and a geodesic from $e$ to $X_n$ is bounded linearly in $l$.
Assuming that $X_n$ converges to a $N$-Morse point $X_\infty$ and letting $n$ tend to infinity, we get that the distance between $X_k$ and an infinite geodesic from $e$ to $X_\infty$ is bounded.

To sum-up, for every $k$, with probability at least $\eta/2$, $X_n$ converges to a $N$-Morse point $X_\infty$ and $X_k$ is at distance at most $D$ from a geodesic from $e$ to $X_\infty$, for some constant $D$.

To conclude, we use the results in \cite{SistoTaylor} which state that the random walk spends essentially a logarithmic time in non-maximal domains (if $\Gamma$ is assumed to be hierarchically hyperbolic) or in parabolic subgroups (if $\Gamma$ is assumed to be relatively hyperbolic).
Precisely, \cite[Theorem~2.3]{SistoTaylor} shows that for all large enough $n$, there exists a subspace $U\subset \Gamma$ such that the probability that $d_U(\pi_U(e),\pi_U(X_n)) \geq C^{-1} \log n$ is at least $1-\eta/4$, where
$U\in \mathfrak{S}\setminus \{\mathbf{S}\}$ if $\Gamma$ is hierarchically hyperbolic, or
$U$ is a parabolic subgroup if $\Gamma$ is relatively hyperbolic.

On the other hand, recall that Proposition~\ref{ABDTheoremE} and Lemma~\ref{relativelyhyperbolicMorseboundedprojections} show that a $N$-Morse geodesic has $D'$-bounded projections on such a $U$, where $D'$ only depends on $N$.
In particular, if the distance from a point $g$ to a $N$-Morse geodesic $\alpha$ is bounded, then $d_U(\pi_U(e),\pi_U(g))$ also is bounded.
This proves that for every $k$, with probability at least $\eta/2$, $d_U(\pi_U(e),\pi_U(X_k))$ is uniformly bounded.

Finally, for large enough n, with probability at least $\eta/4$, for every $U$ as above, $d_U(\pi_U(e),\pi_U(X_n))$ is uniformly bounded while $d_U(\pi_U(e),\pi_U(X_n)) \geq C^{-1} \log n$ for some $U$. This is a contradiction.
We thus proved Theorem~\ref{maintheoremmeasureMorse}. \qed

\bibliographystyle{plain}
\bibliography{MorseMartin}

\end{document}